\documentclass[12pt]{amsart}
\usepackage{amsmath}
\usepackage{eucal}
\usepackage{amssymb}
\usepackage[all]{xy}
\usepackage{longtable}
\usepackage[osf,sc]{mathpazo}
\usepackage[titletoc]{appendix}
\usepackage{color}
\usepackage{multirow,bigdelim}
\usepackage{stmaryrd}
\usepackage{verbatim}
\usepackage{mathrsfs}

\allowdisplaybreaks

\newtheorem{theorem}[equation]{Theorem}
\newtheorem{lemma}[equation]{Lemma}
\newtheorem{proposition}[equation]{Proposition}
\newtheorem{corollary}[equation]{Corollary}

\newtheorem{definition}[equation]{Definition}

\newtheorem{claim}[equation]{Claim}

\theoremstyle{remark}
\newtheorem{remark}[equation]{Remark}

\numberwithin{equation}{subsection}

\allowdisplaybreaks[1]

\newcommand{\DD}{\mathbb{D}}

\newcommand{\bu}{\mathbf{u}}

\DeclareMathAlphabet{\matheur}{U}{eur}{m}{n}

\newcommand{\fs}{\mathfrak{s}}

\newcommand{\DDConv}{\DD^{\text{\fontsize{6pt}{0pt}\selectfont \rm{Conv}}}}
\newcommand{\DDDef}{\DD^{\text{\fontsize{6pt}{0pt}\selectfont \rm{Def}}}}

\DeclareMathOperator{\Aut}{Aut} 
\DeclareMathOperator{\Ker}{Ker} 
\DeclareMathOperator{\Mat}{Mat}

 \DeclareMathOperator{\wt}{wt}
  
\DeclareMathOperator{\Li}{Li}    

\DeclareMathOperator{\ad}{ad}

\DeclareMathOperator{\dep}{dep}

\newcommand{\tr}{\mathrm{tr}}

\begin{document}

\title[]{{\large{A\MakeLowercase{ $v$-adic variant of }A\MakeLowercase{nderson}-B\MakeLowercase{rownawell}-P\MakeLowercase{apanikolas linear independence criterion and its application}}}}

\author{Yen-Tsung Chen}
\address{Department of Mathematics, Pennsylvania State University, University Park, PA 16802, U.S.A.}

\email{ytchen.math@gmail.com}

\thanks{}


\date{\today}

\begin{abstract} 
    For a finite place $v$ of the rational function field $k$ over a finite field, we follow the approach in \cite[Sec.~3]{ABP04} closely to formulate and prove a $v$-adic linear independence criterion. Let $\overline{k}$ be a fixed algebraic closure of $k$. When the finite place $v$ is of degree one, we show that all $\overline{k}$-linear relations among $v$-adic Carlitz multiple polylogarithms at algebraic points arise from $k$-linear relations among these values of the same weight. As an application, we establish a function field analogue of Furusho-Yamashita's conjecture for $v$-adic multiple zeta values whenever the degree of the place $v$ is one.
\end{abstract}


\maketitle

\tableofcontents

\section{Introduction}
    \subsection{Motivation}
    For an $n$-tuple $\mathfrak{n}=(n_1,\dots,n_r)\in\mathbb{Z}_{>0}^r$, the $\mathfrak{n}$-th multiple polylogarithm is defined by the following series
    \[
        \mathscr{L}_{\mathfrak{n}}(z_1,\dots,z_r):=\sum_{m_1>\cdots>m_r>0}\frac{z_1^{n_1}\cdots z_r^{n_r}}{m_1^{n_1}\cdots m_r^{n_r}}\in\mathbb{Q}\llbracket z_1,\dots,z_r\rrbracket.
    \]
    As a complex analytic function, it converges when $|z_i|<1$ for each $1\leq i\leq r$ and $z_i\in\mathbb{C}$. Since $\mathscr{L}_1(z)=-\log(1-z)$, the study of linear independence among multiple polylogarithms at algebraic points dates back to Baker's celebrated theorem on linear forms in logarithms, which asserts that given any non-zero algebraic numbers $\alpha_1,\dots,\alpha_n\in\overline{\mathbb{Q}}$ such that the quantities $\log(\alpha_1),\dots,\log(\alpha_n)$ are linearly independent over $\mathbb{Q}$, then $1,\log(\alpha_1),\dots,\log(\alpha_n)$ are linearly independent over $\overline{\mathbb{Q}}$. Several results along this direction were established by using Diophantine approximation or tools from transcendence theory. It is remarkable that if $n_1>1$ then the special value of the $\mathfrak{n}$-th multiple polylogarithm at $z_1=\cdots=z_r=1$ recovers the \emph{multiple zeta values}
    \[
        \zeta(n_1,\dots,n_r):=\sum_{m_1>\cdots>m_r>0}\frac{1}{m_1^{n_1}\cdots m_r^{n_r}}\in\mathbb{R}^\times.
    \]
    Note that this value generalizes the Riemann zeta function at positive integers, and have been investigated in recent decades due to its interesting properties and connections to various topics. We refer reader to \cite{BGF19,Zha16} for more details.

    Let $\mathbb{F}_q$ be the finite field with $q$ elements, for $q$ a power of a prime number $p$. Let $A:=\mathbb{F}_q[\theta]$ be the polynomial ring with variable $\theta$ and $k:=\mathbb{F}_q(\theta)$ be its field of fractions. Let $|\cdot|_\infty$ be the normalized non-archimedean absolute value on $k$ so that $|f/g|_\infty:=q^{\deg_{\theta}(f)-\deg_{\theta}(g)}$ for     $f,g\in A$ with $g\neq 0$. We denote by $k_\infty$ the {$\infty$-adic} completion of $k$ with respect to $|\cdot|_\infty$. We further set $\mathbb{C}_\infty$ to be the completion of a fixed algebraic closure $\overline{k_\infty}$ and we fix $\overline{k}$ to be the algebraic closure of $k$ inside $\mathbb{C}_\infty$.
    
    In the case of global function field in positive characteristic, Chang \cite{Cha14} introduced the \emph{Carlitz multiple polylogarithms}, abbreviated as CMPLs, generalizing the notion of \emph{Carlitz polylogarithms} studied by Anderson and Thakur in \cite{AT90}.
    To be more precise, let $L_0:=1$ and $L_i:=(\theta-\theta^q)\cdots(\theta-\theta^{q^i})$ for any $i\geq 1$. For an index $\fs=(s_{1},\dots,s_{r})\in \mathbb{Z}_{>0}^{r}$, the $\fs$-th CMPL is defined by the following series (see~\cite{Cha14}):
    \begin{equation}\label{Def:CMPL}
        \Li_\fs(z_1,\dots,z_r):=\underset{i_1>\cdots>i_r\geq 0}{\sum}\frac{z_1^{q^{i_1}}\dots z_r^{q^{i_r}}}{L_{i_1}^{s_1}\cdots L_{i_r}^{s_r}}\in k\llbracket z_1,\cdots,z_r \rrbracket.
    \end{equation}
    Note that $\Li_{1}(z)=\log_C(z)$ is the \emph{Carlitz logarithm}. Its study was initiated by Carlitz in \cite{Car35}. The weight and the depth of the presentation of $\Li_\fs(z_1,\dots,z_r)$ are defined by $\wt(\fs):=s_1+\cdots+s_r$ and $\dep(\fs):=r$ respectively.
    
    Note that $\Li_\fs(z_1,\dots,z_r)$ converges $\infty$-adically on the set (see \cite[Rem.~5.1.4]{Cha14})
    \begin{equation}\label{Def:Conv._Domain_infty}
        \DDConv_{\fs,\infty}:=\{(z_1,\dots,z_r)\in\mathbb{C}_\infty^r\mid |z_i|_\infty<q^{\frac{s_iq}{q-1}}\}.
    \end{equation}
    Let $\mathcal{L}$ be the $k$-vector space spanned by $1$ and $\Li_\fs(\bu)$ with $\bu\in(\overline{k}^\times)^r\cap\DDConv_{\fs,\infty}$. 
    For $w\geq 1$, let $\mathcal{L}_w$ be the $k$-vector space spanned by $\Li_\fs(\bu)$ with $\bu\in(\overline{k}^\times)^r\cap\DDConv_{\fs,\infty}$ and $\wt(\fs)=w$. It is known that $\mathcal{L}$ forms a $k$-algebra with the fact that $\mathcal{L}_{w_1}\mathcal{L}_{w_2}\subset\mathcal{L}_{w_1+w_2}$ by the stuffle relations (see \cite[Section~5.2]{Cha14}). 
    A linear independence result of CMPLs at algebraic points was established by Chang \cite{Cha14}.
    \begin{theorem}[{\cite[Thm.~5.4.3]{Cha14}}]\label{Thm:Cha14_CMPLs}
        Let $\overline{\mathcal{L}}$ be the $\overline{k}$-algebra generated by $\Li_\fs(\bu)$ with $\bu\in(\overline{k}^\times)^r\cap\DDConv_{\fs,\infty}$ and for $w\geq 1$ let $\overline{\mathcal{L}}_{w}$ be the $\overline{k}$-vector space spanned by $\Li_\fs(\bu)$ of weight $w$ with $\bu\in(\overline{k}^\times)^r\cap\DDConv_{\fs,\infty}$. The following assertions hold.
        \begin{enumerate}
            \item $\overline{\mathcal{L}}=\overline{k}\oplus\left(\oplus_{w\geq 1}\overline{\mathcal{L}}_{w}\right)$ forms a graded $\overline{k}$-algebra.
            \item The canonical map $\overline{k}\otimes_k\mathcal{L}\to \overline{\mathcal{L}}$ is a bijection.
        \end{enumerate}
    \end{theorem}
    The $\infty$-adic multiple zeta values, abbreviated as MZVs, were defined by Thakur in \cite{Tha04}, generalizing the notion of the Carlitz zeta values in \cite{Car35}. More precisely, for any index $\fs=(s_1,\dots,s_r)\in\mathbb{Z}_{>0}^r$, the $\infty$-adic MZV at $\fs$ is defined by the series
    \begin{equation}\label{Eq:MZV}
        \zeta_A(\fs):=\sum\frac{1}{a_1^{s_1}\cdots a_r^{s_r}}\in k_\infty,
    \end{equation}
    where $(a_1,\dots,a_r)\in A^r$ with $a_i$ monic and $\deg_\theta a_i$ strictly decreasing. The weight and the depth of the presentation $\zeta_A(\fs)$ are defined by $\wt(\fs)$ and $\dep(\fs)$ respectively. It was shown by Thakur in \cite{Tha09} that $\zeta_A(\fs)$ is non-vanishing for any index $\fs=(s_1,\dots,s_r)\in\mathbb{Z}^r_{>0}$.
    Let $\mathcal{Z}$ be the $k$-vector space spanned by $1$ and all $\infty$-adic MZVs. For $w\geq 1$, let $\mathcal{Z}_w$ be the $k$-vector space spanned by MZVs of weight $w$. Thakur showed in \cite{Tha10} that $\mathcal{Z}$ forms a $k$-algebra by sum-shuffle relations. In particular, one has $\mathcal{Z}_{w_1}\mathcal{Z}_{w_2}\subset\mathcal{Z}_{w_1+w_2}$ for $w_1\geq 1$ and $w_2\geq 1$. As a consequence of Theorem.~\ref{Thm:Cha14_CMPLs}, Chang successfully established the following function field analogue of Goncharov's direct sum conjecture \cite{Gon97} for $\infty$-adic MZVs.
    \begin{theorem}[{\cite[Thm.~2.2.1]{Cha14}}]\label{Thm:Cha14}
        Let $\overline{\mathcal{Z}}$ be the $\overline{k}$-algebra generated by all $\infty$-adic MZVs and for $w\geq 1$ let $\overline{\mathcal{Z}}_w$ be the $\overline{k}$-vector space spanned by $\infty$-adic MZVs of weight $w$. Then we have
        \begin{enumerate}
            \item $\overline{\mathcal{Z}}=\overline{k}\oplus\left(\oplus_{w\geq 1}\overline{\mathcal{Z}}_w\right)$ forms a graded $\overline{k}$-algebra.
            \item The canonical map $\overline{k}\otimes_k\mathcal{Z}\to \overline{\mathcal{Z}}$ is a bijection.
        \end{enumerate}
    \end{theorem}
    The primary goal in this article is to establish the $v$-adic analogue of Theorem~\ref{Thm:Cha14_CMPLs} and Theorem~\ref{Thm:Cha14} for any finite place $v$ of degree one.

\subsection{The main result}    
    
    
    
    
    Given a monic irreducible polynomial $\varpi_v\in A$ which corresponds to a finite place $v$ of $k$. Let $|\cdot|_v$ be the associated norm to the finite place $v$ such that $|\varpi_v|_v=q_v^{-1}$ where $q_v:=|A/vA|$ is the cardinality of the residue field. We set $k_v$ to be the completion of $k$ with respect to $|\cdot|_v$ and define $\mathbb{C}_v$ to be the completion of a fixed algebraic closure of $k_v$. Throughout this article, we fix an embedding from $\overline{k}$ into $\mathbb{C}_v$ once a finite place $v$ is given.
    
    We denote by $\Li_\fs(z_1,\cdots,z_r)_v$ when we regard this power series as a $v$-adic function. Note that $\Li_\fs$ converges on the set (see \cite[Sec.~2.2]{CM19})
    \begin{equation}\label{Def:Conv._Domain_v}
        \DDConv_{\fs,v}:=\{(z_1,\dots,z_r)\in\mathbb{C}_v^r\mid |z_1|_v<1,|z_\ell|_v\leq 1\mbox{ for }2\leq\ell\leq r\}.
    \end{equation}
    Let $\mathcal{L}_v$ be the $k$-vector space spanned by $1$ and $\Li_\fs(\bu)_v$ with $\bu\in(\overline{k}^\times)^r\cap\DDConv_{\fs,v}$. For $w\geq 1$, let $\mathcal{L}_{v,w}$ be the $k$-vector space spanned by $\Li_\fs(\bu)_v$ with $\bu\in(\overline{k}^\times)^r\cap\DDConv_{\fs,v}$ and $\wt(\fs)=w$. It is known that $\mathcal{L}_v$ forms a $k$-algebra and $\mathcal{L}_{v,w_1}\mathcal{L}_{v,w_1}\subset\mathcal{L}_{v,w_1+w_2}$ by the stuffle relations (c.f. \cite[Section~5.2]{Cha14}, \cite[Prop.~2.2.3]{CCM22}). 
        
    The main result of the present paper can be stated as follows.
    \begin{theorem}\label{Thm:CMPLs}
        Let $\overline{\mathcal{L}}_v$ be the $\overline{k}$-algebra generated by $\Li_\fs(\bu)_v$ with $\bu\in(\overline{k}^\times)^r\cap\DDConv_{\fs,v}$ and for $w\geq 1$ let $\overline{\mathcal{L}}_{v,w}$ be the $\overline{k}$-vector space spanned by $\Li_\fs(\bu)_v$ of weight $w$ with $\bu\in(\overline{k}^\times)^r\cap\DDConv_{\fs,v}$. If the finite place $v$ is of degree one, then the following assertions hold.
        \begin{enumerate}
            \item $\overline{\mathcal{L}}_v=\overline{k}\oplus\left(\oplus_{w\geq 1}\overline{\mathcal{L}}_{v,w}\right)$ forms a graded $\overline{k}$-algebra.
            \item The canonical map $\overline{k}\otimes_k\mathcal{L}_v\to \overline{\mathcal{L}}_v$ is a bijection.
        \end{enumerate}
    \end{theorem}
    
    An immediate consequence of Theorem \ref{Thm:CMPLs} is the following.
    
    \begin{corollary}
        Let $\fs=(s_1,\dots,s_r)\in\mathbb{Z}_{>0}^{r}$ be an index and $\bu=(u_1,\dots,u_r)\in (\overline{k}^\times)^r\cap\DDConv_{\fs,v}$. If the finite place $v$ is of degree one, then $\Li_\fs(\bu)_v$ is either zero or transcendental over $k$.
    \end{corollary}
    
    To describe an application of Theorem~\ref{Thm:CMPLs}, we briefly review the definition of $v$-adic MZVs introduced by Chang and Mishiba in \cite{CM21}. 
    Based on the formula of Chang \cite{Cha14} expressing $\infty$-adic MZVs in terms of CMPLs at integral points, Chang and Mishiba introduced the Carlitz multiple  star polylogarithms, abbreviated as CMSPLs, given in (\ref{Def:CMSPL}) and derived the formula of $\infty$-adic MZVs in terms of CMSPLs at integral points.
    This formula generalized the result of Anderson and Thakur in the depth one case \cite{AT90}.
    Let $\mathfrak{p}\in\mathbb{Z}$ be a prime number. Inspired by the definition of Furusho's $\mathfrak{p}$-adic MZVs \cite{Fur04} using Coleman's integration theory \cite{Col82} to analytically continue the classical polylogarithms $\mathfrak{p}$-adically, Chang and Mishiba regarded CMSPLs as a $v$-adic function defined by the same series and then used the logarithmic interpretation of CMSPLs to extend the defining domain so that all integral points can be included. 
    Then they defined the $v$-adic MZVs $\zeta_A(\fs)_v$ by using the same formula of $\infty$-adic MZVs for any index $\fs\in\mathbb{Z}_{>0}^r$ (see Theorem~\ref{Thm:MZVs_to_CMSPLs} and Definition~\ref{Def:vMZV}).
    
    Let $\mathcal{Z}_v$ be the $k$-vector space spanned by $1$ and all $v$-adic MZVs. For $w\geq 1$ let $\mathcal{Z}_{v,w}$ be the $k$-vector space spanned by $v$-adic MZVs of weight $w$. It is shown in \cite{CCM22} that $v$-adic MZVs forms a $k$-algebra with multiplication law given by sum-shuffle relations as the case of $\infty$-adic MZVs. Moreover, we have $\mathcal{Z}_{v,w_1}\mathcal{Z}_{v,w_2}\subset \mathcal{Z}_{v,w_1+w_2}$ for $w_1\geq 1$ and $w_2\geq 1$. As an application of Theorem~\ref{Thm:CMPLs}, we formulate and prove a function field analogue of Furusho-Yamashita's conjecture \cite[Conj.~5]{Yam10} when $v$ is of degree one.
    \begin{theorem}\label{Main_Thm}
        Let $\overline{\mathcal{Z}}_v$ be the $\overline{k}$-algebra generated by all $v$-adic MZVs and for $w\geq 1$ let $\overline{\mathcal{Z}}_{v,w}$ be the $\overline{k}$-vector space spanned by $v$-adic MZVs of weight $w$. If $v$ is of degree one, then the following assertions hold.
        \begin{enumerate}
            \item $\overline{\mathcal{Z}}_v=\overline{k}\oplus\left(\oplus_{w\geq 1}\overline{\mathcal{Z}}_{v,w}\right)$ forms a graded $\overline{k}$-algebra.
            \item The canonical map $\overline{k}\otimes_k\mathcal{Z}_v\to \overline{\mathcal{Z}}_v$ is a bijection.
        \end{enumerate}
    \end{theorem}
    
    For an positive integer $s$, we say that $s$ is $q$-even if $s$ is divisible by $q-1$ and otherwise it is called $q$-odd. In the case of $r=1$, Goss showed in \cite{Gos79} that $\zeta_A(s)_v=0$ for each $q$-even $s\in\mathbb{Z}_{>0}$, and Yu proved in \cite{Yu91} that $\zeta_A(s)_v$ is transcendental over $k$ for each $q$-odd $s\in\mathbb{Z}_{>0}$.
    An immediate consequence of Theorem \ref{Main_Thm} is the higher depth generalization of the result of Goss and Yu when $v$ is of degree one.
    
    \begin{corollary}
        Let $\fs=(s_1,\dots,s_r)\in\mathbb{Z}_{>0}^{r}$ be an index. If the finite place $v$ is of degree one, then $\zeta_A(\fs)_v$ is either zero or transcendental over $k$.
    \end{corollary}
    
    \subsection{Strategy and Organization}
    Let $\mathfrak{p}\in\mathbb{Z}$ be a prime number. In the $\mathfrak{p}$-adic transcendence theory in characteristic zero or $v$-adic transcendence theory over function fields in positive characteristic, one might work with functions which are defined only on a small region, even though their complex or $\infty$-adic counterpart are entire. 
    This problem also occurs in our study of Theorem \ref{Thm:CMPLs}. For instance, the function used in the proof of \cite{Cha14} is defined by a power series which is $\infty$-adically entire with algebraic coefficients (see \cite[Lem.5.3.1]{Cha14} for more details). 
    But when we regard it as a $v$-adic function defined by the same series, it is not $v$-adically entire anymore. This problem is the main difficulty to establish Theorem~\ref{Thm:CMPLs}.
    
    The strategy and organization are given as follows. In Section 2, we prove a linear independence criterion which can be regarded as $v$-adic analogue of \cite[Thm.~3.1.1]{ABP04}. As most arguments of the proof are parallel to the $\infty$-adic version, all the details are given in the appendix. In Section $3$, we generalize the notion of MZ (multizeta) property which was studied in \cite{Cha14} and we develop the essential tools for our later use. Section $4$ is occupied by the most technical part. We construct a \emph{$v$-adic entire} function so that the desired value appears in its specialization at a certain algebraic point. Since our construction heavily relies on the specific form of the corresponding irreducible polynomial $\varpi_v$, the restriction on the degree of $v$ is necessary for our argument. In Section $4$, we give the detailed proof of Theorem~\ref{Thm:CMPLs} and Theorem~\ref{Main_Thm}.

    
    
    \section*{Acknowledgement}
    The content of the present article is based on author's Ph.D. dissertation, and the author is grateful to his advisor C.-Y. Chang for fruitful discussions. The author also appreciates O. Gezmis, Y. Mishiba, and C. Namoijam for valuable comments and suggestions. Parts of this work were done when the author visited Osaka University in the summer of 2019 and Texas A\&M University from August 2021-May 2022. The former one was partially supported by the College of Science at National Tsing Hua University and the later one was funded by the Ministry of Science and Technology. The author thanks S. Yasuda, Y. Sugiyama, and M. Papanikolas for their hospitality.

\section{Preliminaries}
\subsection{Notation and the setup}
    \begin{itemize}
		\setlength{\leftskip}{0.8cm}
        \setlength{\baselineskip}{18pt}
		\item[$\mathbb{F}_q$] :=  a fixed finite field with $q$ elements, for $q$ a power of a prime number $p$.
		\item[$\infty$] :=  the point at infinity of the projective line $\mathbb{P}^1_{/\mathbb{F}_q}$.
		\item[$A$] :=  $\mathbb{F}_q[\theta]$, the ring of rational functions of $\mathbb{P}^1_{/\mathbb{F}_q}$ regular away from $\infty$.
		\item[$k$] :=  $\mathbb{F}_q(\theta)$, the function field of $\mathbb{P}^1_{/\mathbb{F}_q}$.
		\item[$|\cdot|_\infty$] := the nomalized absolute value on $k$ associated to $\infty$ so that $|\theta|_\infty=q$.
		\item[$k_\infty$] :=  the completion of $k$ at the place $\infty$.
		\item[$\mathbb{C}_\infty$] :=  the $\infty$-adic completion of a fixed algebraic closure of $k_\infty$.
		\item[$\overline{k}$] := a fixed algebraic closure of $k$ inside $\mathbb{C}_\infty$.
		\item[$\mathcal{E}$] := $\{f=\sum_{i=0}^\infty a_it^i\in\overline{k}\llbracket t\rrbracket\mid \lim_{i\to\infty}\sqrt[i]{|a_i|_\infty}=0,~[k_\infty(f_0,f_1,\dots):k_\infty]<\infty\}$.
	\end{itemize}
 
    Throughout this article, we denote by $v$ a fixed finite place of $k$ and $\varpi_v\in A$ the associated monic irreducible polynomial with degree $\epsilon_v$. Let $k_v$ be the completion of $k$ at the place $v$ and $\mathbb{C}_v$ be the completion of a fixed algebraic closure $\overline{k_v}$. We denote by $|\cdot|_v$ the normalized $v$-adic norm so that $|\varpi_v|_v=q_v^{-1}$, where $q_v:=q^{\epsilon_v}$. We fix an embedding $\iota_v:\overline{k}\hookrightarrow\mathbb{C}_v$, and we regard $\overline{k}$ as a subfield of $\mathbb{C}_v$ via this embedding. In particular, for $\alpha\in\overline{k}$, we realize $|\alpha|_v:=|\iota(\alpha)|_v$. The $v$-adic Tate algebra is denoted by
    \[
        \mathbb{T}_v:=\{\sum_{i\geq 0}a_it^i\in\mathbb{C}_v\llbracket t \rrbracket\mid\lim_{i\to\infty}|a_i|_v=0\},
    \]
    which is complete with respect to the $v$-adic Gauss norm $\|\cdot\|_v$ defined by 
    \[
        \|\sum_{i\geq 0}a_it^i\|_v:=\sup_{i\geq 0}|a_i|_v.
    \]
    The ring of $v$-adic entire series is denoted by 
    \[
        \mathcal{E}_v:=\{\sum_{i=0}^{\infty}a_it^i\in\overline{k}\llbracket t \rrbracket\mid \lim_{i\to\infty}|a_i|_v^{1/i}=0,[k_v(a_0,a_1,\dots):k_v]<\infty\}.
    \]
    By abuse of the notation, for $B=(B_{ij})\in\Mat_{m\times n}(\mathbb{T}_v)$, we set $\|B\|_v:=\max_{i,j}\{\|B_{ij}\|_v\}$.

\subsection{ABP criterion}
    In \cite{ABP04}, Anderson, Brownawell, and Papanikolas established a linear independence criterion for special values in global function fields. 
    Later on, Chang generalized this criterion into a more general situation in \cite{Cha09} which can be regarded as a function field analogue of the Siegel-Shidlovskii theorem for classical $E$-functions. Many applications such as determining all algebraic relations among special $\Gamma$-values \cite{ABP04}, algebraic independence of Carlitz logarithms at algebraic points \cite{Pap08} and algebraic independence of zeta values in positive characteristic \cite{CY07,CPY11} can be achieved.

    For the convenience of our later use, given an algebraically closed field $\mathbb{K}$ that contains $\mathbb{F}_q$, consider the Laurent series field $\mathbb{K}(\!(t)\!)$. For $n\in\mathbb{Z}$, we define the $n$-fold Frobenius twisting by
    \begin{align*}
        \mathbb{K}(\!(t)\!)&\to\mathbb{K}(\!(t)\!)\\
        f=\sum_{i>N}a_it^i&\mapsto f^{(n)}:=\sum_{i>N}a_i^{q^n}t^i.
    \end{align*}
    In the present article, we only consider the case of $\mathbb{K}=\mathbb{C}_\infty$ or $\mathbb{K}=\mathbb{C}_v$. The refined version of the ABP linear independence criterion can be stated as follows.
    
    \begin{theorem}[{\cite[Thm.~3.1.1]{ABP04}},~{\cite[Thm.~1.2.(1)]{Cha09}}]\label{Thm:ABP}
        Fix a square matrix $\Phi(t)\in\Mat_{\ell}(\overline{k}[t])$ and a vector $\psi(t)\in\Mat_{\ell\times 1}(\mathcal{E})$
        satisfying $\psi^{(-1)}=\Phi\psi$.
        Let $\gamma\in\overline{k}^\times\setminus\overline{\mathbb{F}}_q^\times$ be such that $\det\Phi(\gamma^{(-i)})\neq 0$ for all $i\in\mathbb{Z}_{>0}.$ Then, for any $\rho\in\Mat_{1\times\ell}(\overline{k})$ such that $\rho\psi(\gamma)=0$, there exists $P(t)\in\Mat_{1\times\ell}(\overline{k}[t])$ such that $P(\gamma)=\rho$ and $P\psi=0$.
    \end{theorem}
    
    We find that the same spirit of Theorem~\ref{Thm:ABP} can be carried out to the following $v$-adic setting, which can be used for our purpose.
    
    \begin{theorem}\label{Thm:vABP}
        Fix a square matrix $\Phi(t)\in\Mat_{\ell}(\overline{k}[t])$ and a vector $\psi(t)\in\Mat_{\ell\times 1}(\mathcal{E}_v)$ satisfying $\psi^{(-1)}=\Phi\psi.$
        Let $\gamma\in\overline{k}^\times\setminus\overline{\mathbb{F}}_q^\times$ such that $\det\Phi(\gamma^{(-1)})\neq 0$ for all $i\in\mathbb{Z}_{>0}$. Then, for $\rho\in\Mat_{1\times\ell}(\overline{k})$ such that $\rho\psi(\gamma)=0$, there exists $P(t)\in\Mat_{1\times\ell}(\overline{k}[t])$ such that
        $P(\gamma)=\rho$ and $P\psi=0$.
    \end{theorem}
    
    As we pointed out in the introduction, the structure of the proof for Theorem~\ref{Thm:vABP} is essentially the same as the proof of Theorem~\ref{Thm:ABP}. To be self-contained, we provide detailed proof in the appendix.

\section{Main results and its applications}
\subsection{The graded space of MPLs}
    As an application of Theorem \ref{Thm:ABP}, Chang introduced the notion of MZ (multizeta) property in \cite{Cha14} and developed a linear independence criterion of values with this property. In this subsection, we follow Chang's strategy \cite{Cha14} to formulate a generalized version of MZ property and then use Theorem \ref{Thm:vABP} to establish a $v$-adic linear independence criterion. To begin with, we introduce the following series which plays a crucial role for our later study.
    
    \begin{definition}
        Let $\alpha\in\mathbb{C}_v^\times$ with $|\alpha|_v<1$. We define the $v$-adic Anderson-Thakur series associated to $\alpha$ to be the power series
        \[
            \Omega_\alpha(t):=\prod_{i\geq 1}(1-\alpha^{q^i}t)\in\mathbb{C}_v\llbracket t \rrbracket.
        \]
    \end{definition}
    
    \begin{remark}
        Note that in the $\infty$-adic case, if we choose $\alpha=\varpi_\infty:=1/\theta$, then the power series 
        \[
            \Omega_{\varpi_\infty}(t):=\prod_{i=1}^\infty(1-\frac{t}{\theta^{q^i}})\in\mathcal{E}
        \]
        induces an entire function on $\mathbb{C}_\infty$.
        It is essentially the same function
        \[
            \Omega(t):=(-\theta)^{-q/(q-1)}\Omega_{\varpi_\infty}(t)\in\mathcal{E}
        \]
        defined in \cite[Sec.~3.1.2]{ABP04}. Note that
        \[
            \frac{-1}{\Omega(\theta)}=\frac{-1}{(-\theta)^{-q/(q-1)}\Omega_{\varpi_\infty}(\varpi_\infty^{-1})}=:\Tilde{\pi}
        \]
        is a period of the Carlitz module (see \cite[Cor.5.1.4]{ABP04}).
    \end{remark}
    
    The basic property of $\Omega_\alpha$ is established in the following proposition.
        
    \begin{proposition}[cf.~{\cite[Sec.~3.1.2]{ABP04}}]\label{Prop:Basic_Ex}
        Let $\alpha\in\overline{k}^\times$ such that $|\alpha|_v<1$. Then we have
        \[
            \Omega_\alpha\in\mathcal{E}_v~\mbox{ and }~\Omega_\alpha^{(-1)}=(1-\alpha t)\Omega_\alpha.
        \]
        In particular, $\Tilde{\pi}_\alpha:=\Omega_\alpha(\alpha^{-1})\in\overline{k}_v^\times$ is transcendental over $k$.
    \end{proposition}
        
    \begin{proof}
        Note that for any sequence $\{a_n\}_{n\geq 1}\subset\mathbb{C}_v$, $\prod_{n\geq 1}(1+a_n)$ converges in $\mathbb{C}_v$ if and only if $\lim_{n\to\infty}|a_n|_v=0$. Since for each $x\in\mathbb{C}_v$ we have $\lim_{i\to\infty}|\alpha^{q_v^i}x|_v=0$, $\Omega_\alpha$ defines an entire function on $\mathbb{C}_v$ . 
        On the other hand, if we express $\Omega_\alpha(t)=\sum_{i\geq 0}c_it^i$ for $c_i\in \overline{k}_v^\times$, then the functional equation of $\Omega_\alpha$ implies that
        \[
            c_0^{q^{-1}}+\alpha^{-1}c_0=0\mbox{ and } c_i^{q^{-1}}+\alpha^{-1}c_i=c_{i-1}\mbox{ for }i>0.
        \]
        This implies that $\Omega_\alpha\in\overline{k}\llbracket t\rrbracket$ and hence the first assertion $\Omega_\alpha\in\mathcal{E}_v$ follows. 
        Note that the difference equation $\Omega_\alpha^{(-1)}=(1-\alpha t)\Omega_\alpha$ follows immediately from direct computations. 
        
        To prove the transcendence of $\Tilde{\pi}_\alpha$, we suppose on the contrary that $\Tilde{\pi}_\alpha\in\overline{k}$. Thus, there exists $\{\rho_i\}_{i=0}^\ell\subset k$ such that 
        $$\sum_{i=0}^\ell\rho_i\Tilde{\pi}_\alpha^i=0,~(\rho_0\rho_\ell\neq 0,~\ell>0).$$ 
        Let 
        $$\rho:=(\rho_0,\dots,\rho_\ell)\in\Mat_{1\times (\ell+1)}(\overline{k})$$
        and 
        $$\psi(t):=(1,\Omega_\alpha,\dots,\Omega_\alpha^\ell)^\tr\in\Mat_{(\ell+1)\times 1}(\mathcal{E}_v).$$ 
        Then $\rho\psi(\alpha^{-1})=0$ and $\psi^{(-1)}=\Phi\psi$, where
        $$\Phi=
        \begin{pmatrix}
            1 & 0 & \cdots & 0 \\
            0 & (1-\alpha t) &  & \vdots\\
            \vdots &  & \ddots & 0\\
            0 & \cdots & 0 & (1-\alpha t)^\ell
        \end{pmatrix}
        \in\Mat_{(\ell+1)\times(\ell+1)}(\overline{k}[t]).$$ 
        By Theorem~\ref{Thm:vABP}, there exists $P(t)=(P_0(t),\dots,P_\ell(t))\in\Mat_{1\times(\ell+1)}(\overline{k}[t])$ such that 
        $$P_i(\alpha^{-1})=\rho_i~\mbox{ and }~\sum_{i=0}^\ell P_i\Omega_\alpha^i=0.$$ 
        Since $\Omega_\alpha(t)$ has a simple zero at $\alpha^{-q^i}$ for all $i\in\mathbb{Z}_{>0}$, specializing at $t=\alpha^{-q^{i}}$ of $\sum_{i=0}^\ell P_i\Omega_\alpha^i=0$, we obtain that $P_0(\alpha^{-q^i})=0$ for all $i\in\mathbb{Z}_{>0}$. Thus, $P_0(t)$ vanishes identically and it contradicts to the assumption $\rho_0=P_0(\alpha^{-1})\neq 0$.
    \end{proof}
    
    Now, we are ready to formulate a generalized version of the MZ property.
    
    \begin{definition}[cf.~{\cite[Def.~3.4.1]{Cha14}}]\label{Def:MZ}
        Let $0\neq f(t)\in\overline{k}[t]$, $\alpha\in\overline{k}^\times$ such that $|\alpha|_v<1$. A nonzero element $Z\in\overline{k}_v^\times$ is said to have the $f(t)$-type MPL~(multipolylog)~property of weight $w$ with respect to $\alpha$ if there exists $c\in\overline{k}^\times$, $\Phi(t)\in\Mat_{\ell\times\ell}(\overline{k}[t])$ and $\psi(t)\in\Mat_{\ell\times 1}(\mathcal{E}_v)$ with $\ell\geq 2$ so that
        \begin{enumerate}
            \item $\psi^{(-1)}=\Phi\psi$ and $\Phi$ satisfy the conditions of Theorem~\ref{Thm:vABP} with parameter $\gamma=\alpha^{-1}$.
            \item The last column of $\Phi$ is of the form $(0,\dots,0,f(t))^\tr$.
            \item The first and the last entry of $\psi(\alpha^{-1})$ are $\Tilde{\pi}_\alpha^w$ and $cZ\Tilde{\pi}_\alpha^w$, respectively.
            \item For any positive integer $N$, $\psi(\alpha^{-q^N})$ is of the form $(0,\dots,0,(cZ\Tilde{\pi}_\alpha^w)^{q^N})^\tr$.
        \end{enumerate}
    \end{definition}
    
    In what follows, we introduce the notion of the graded MPL system to study $\overline{k}$-linear relations among the values in $\overline{k}^\times_v$ satisfying the MPL property.
    For column vectors $\psi_1,\dots,\psi_m$ with entries in $\mathbb{C}_v\llbracket t\rrbracket$, we define 
    $$\oplus_{j=1}^m\psi_j:=(\psi_1^\tr,\dots,\psi_m^\tr)^\tr.$$ Finally, for square matrices $\Phi_1,\dots,\Phi_m$ with entries in $\mathbb{C}_v\llbracket t\rrbracket$, we define $\oplus_{j=1}^m\Phi_j$ to be the block diagonal square matrix
    $$
    \begin{pmatrix}
        \Phi_1 &  &  \\
         & \ddots &  \\
         &  & \Phi_m
    \end{pmatrix}.
    $$
    
    \begin{definition}\label{Def:MZ_system}
        Let $\{f_w\}_{w\geq 1}\subset\overline{k}[t]$ be a sequence of nonzero polynomials and $\alpha\in\overline{k}^\times$ such that $|\alpha|_v<1$. Let $V_w$ be a set of elements in $\overline{k}^\times_v$ having the $f_w(t)$-type MPL property of weight $w$ with respect to $\alpha$. We set $G_w$ to be the $k(\alpha)$-vector space spanned by elements in $V_w$ for all $w\geq 1$ and let $G=k(\alpha)+\sum_{w\geq 1}G_w$. Then, the tuple $(G,\{G_w\},\{f_w\},\{V_w\},\alpha)$ is called a $v$-adic graded MPL system.
    \end{definition}
    
    The first goal in this section is to prove the following.
    
    \begin{theorem}\label{Thm:GMZ}
         Let $(G,\{G_w\},\{f_w\},\{V_w\},\alpha)$ be a $v$-adic graded MPL system.
         \begin{enumerate}
             \item Let $\overline{G}_w$ be the $\overline{k}$-vector space spanned by elements in $V_w$ for all $w\geq 1$ and let $\overline{G}=\overline{k}+\sum_{w\geq 1}\overline{G}_w$. Then, $\overline{G}=\overline{k}\oplus_{w\geq 1}\overline{G}_w$ forms a graded $\overline{k}$-vector space, and the canonical map $\overline{k}\otimes_{k(\alpha)} G\to \overline{G}$ is bijective.
             \item If $G_{w_1}G_{w_2}\subset G_{w_1+w_2}$ for all $w_1\geq 1$ and $w_2\geq 1$, then $\overline{G}$ forms a graded $\overline{k}$-algebra.
         \end{enumerate}
    \end{theorem}
    
    An immediate consequence is the following corollary.
    
    \begin{corollary}
        Let $(G,\{G_w\},\{f_w\},\{V_w\},\alpha)$ be a $v$-adic graded MPL system such that for all $w_1\geq 1$ and $w_2\geq 1$ we have $G_{w_1}G_{w_2}\subset G_{w_1+w_2}$. Then, for all $w\geq 1$, every nonzero element in $\overline{G}_w$ is transcendental over $k$.
    \end{corollary}

    To prove Theorem \ref{Thm:GMZ}, we require the following two lemmas, which can be viewed as an extended version of \cite[Sec.~4.1,~4.2]{Cha14} in the $\infty$-adic MZV case. Some technical arguments in the proof are rooted in \cite{Cha14}.
        
    \begin{lemma}\label{Lem:PreGMZ1}
        Let $V$ be a finite set of elements in $\overline{k}_v^\times$ having the $f(t)$-type MPL property of weight $w$ with respect to $\alpha$. Suppose that $V$ is a linearly dependent set over $\overline{k}$. Then $V$ is in fact a linearly dependent set over $k(\alpha)$.
    \end{lemma}
        
    \begin{proof}
        Let $V=\{Z_1,\dots,Z_m\}$. Without loss of generality, we may assume $m\geq 2,~Z_1\in\overline{k}\mbox{-span}\{Z_2,\dots,Z_m\}$, and $V\setminus\{Z_1\}$ is a $\overline{k}$-linear independent set. According to the $f(t)$-type MPL property of weight $w$ with respect to $\alpha$, there exists
        $$\Phi_j(t)\in\Mat_{d_j\times d_j}(\overline{k}[t])\mbox{ and }\psi_j(t)\in\Mat_{d_j\times 1}(\mathcal{E}_v)$$
        associated from $Z_j$ satisfying Definition~\ref{Def:MZ}.
        Let $d:=\sum_{j=1}^md_j$. Consider the matrix $\Phi:=\oplus_{j=1}^{m}\Phi_{j}\in\Mat_{d\times d}(\overline{k}[t])$ and the vector $\psi:=\oplus_{j=1}^{m}\psi_{ij}\in\Mat_{d\times 1}(\mathcal{E}_v)$. Then, we have $\psi^{(-1)}=\Phi\psi$. Moreover, 
        $\psi(\alpha^{-1})$ is of the form
        $$\psi(\alpha^{-1})=\oplus_{j=1}^m(\Tilde{\pi}_\alpha^{w},\star,\dots,\star,c_{j}Z_{j}\Tilde{\pi}_\alpha^{w})^{\tr}$$
        and $\psi(\alpha^{-q^N})$ is of the form
        $$\psi(\alpha^{-q^N})=\oplus_{j=1}^m(0,\dots,0,(c_{j}Z_{j}\Tilde{\pi}_\alpha^{w})^{q^N})^{\tr}$$
        for all positive integers $N$. By using Theorem~\ref{Thm:vABP}, there exist vectors        $$P_j=(P_{j1},\dots,P_{jd_j})\in\Mat_{1\times d_j}(\overline{k}[t])$$
        such that if we put $P=(P_1,\dots,P_m)$, then we have 
        $$P\psi=0,~P_{1d_1}(\alpha^{-1})=1\mbox{ and }P_{ji}(\alpha^{-1})=0\mbox{ for all }1\leq i<m_j.$$ 
        
        Let $\mathbf{P}:=\frac{1}{P_{1d_1}}P$. Then, $\mathbf{P}$ is of the form
        $$\mathbf{P}=(\mathbf{p}_{11},\dots,\mathbf{p}_{1d_1},\dots,\mathbf{p}_{m1},\dots,\mathbf{p}_{md_m})\in\Mat_{1\times d}(\overline{k}(t))$$ 
        where $\mathbf{p}_{1d_1}=1$, and we have        $$\mathbf{P}\psi=0\mbox{ and }\mathbf{p}_{ji}(\alpha^{-1})=0\mbox{ for all }1\leq i<m_j.$$ 
        Subtracting the two equations $\mathbf{P}\psi=0$ and $(\mathbf{P}\psi)^{(-1)}=0$, we have
        $$(\mathbf{P}-\frac{1}{f(t)}\mathbf{P}^{(-1)}\Phi)\psi=0.$$ 
        Since the last column of each matrix $\Phi_j$ is $(0,\dots,0,f(t))^\tr$, we have 
        $$\mathbf{P}-\frac{1}{f(t)}\mathbf{P}^{(-1)}\Phi=(\Tilde{\mathbf{p}}_{11},\dots,\Tilde{\mathbf{p}}_{1d_1},\dots,\Tilde{\mathbf{p}}_{m1},\dots,\Tilde{\mathbf{p}}_{md_m})$$
        where 
        $$\Tilde{\mathbf{p}}_{1d_1}=0\mbox{ and }\Tilde{\mathbf{p}}_{jd_j}=\mathbf{p}_{jd_j}-\mathbf{p}_{jd_j}^{(-1)}\mbox{ for }j=2,\dots,m.$$ 
        
        Now, we claim that $\Tilde{\mathbf{p}}_{jd_j}$ is identically zero for $j=2,\dots,m$. To prove this claim, suppose on the contrary that there exists some $2\leq j\leq m$ such that $\Tilde{\mathbf{p}}_{jd_j}$ is nonzero. Then, we pick $N$ large enough so that all entries of $\mathbf{P}-\frac{1}{f(t)}\mathbf{P}^{(-1)}\Phi$ are regular at $t=\alpha^{-q^N}$, and this can be done because the cardinality of the set $\{\alpha^{-q^N}\}$ is countably infinite. Since $\Tilde{\mathbf{p}}_{jd_j}$ is non-vanishing at $t=\alpha^{-q^N}$, specializing $(\mathbf{P}-\frac{1}{f(t)}\mathbf{P}^{(-1)}\Phi)\psi=0$ at $t=\alpha^{-q^N}$ gives a nontrivial $\overline{k}$-linear relation among $Z_2^{q^N},\dots,Z_m^{q^N}$. By taking the $q^N$-th root, we obtain a nontrivial $\overline{k}$-relation among $Z_2,\dots,Z_m$ and hence a contradiction. Consequently, we have $\mathbf{p}_{jd_j}\in\mathbb{F}_q(t)$ for $j=2,\dots,m$. Note that each entry of $\mathbf{P}$ is regular at $t=\alpha^{-1}$. By specializing $\mathbf{P}\psi=0$ at $t=\alpha^{-1}$, we obtain a nontrivial $F_\alpha$-linear relation among $Z_1,\dots,Z_m$, and the desired result follows.
    \end{proof}
    
    \begin{lemma}\label{Lem:PreGMZ2}
        Let $\{f_1,\dots,f_d\}\subset\overline{k}[t]$ be a finite set of nonzero polynomials, $\{w_1,\dots,w_d\}\subset\mathbb{Z}_{>0}$ be a finite set of distinct positive integers and $\alpha\in\overline{k}^\times$ such that $|\alpha|_v<1$. Let $V_{i}$ be a finite set of values of $\overline{k}^\times_v$ having the $f_i(t)$-type MPL property of weight $w_i$ with respect to $\alpha$ for each $1\leq i\leq d$. If $$\{1\}\cup\bigcup_{i=1}^dV_i$$ is a linearly dependent set over $\overline{k}$, then there exists $1\leq i\leq d$ such that $V_i$ is also a linearly dependent set over $\overline{k}$.
    \end{lemma}
        
    \begin{proof}
        After reordering, we may assume that $w_1>\cdots>w_d$. Since        $$\{1\}\cup\bigcup_{i=1}^dV_i$$
        is a linearly dependent set over $\overline{k}$, without loss of generality, we may assume that there is a nontrivial $\overline{k}$-linear relation between $V_1$ and $\{1\}\bigcup_{i=2}^dV_i$. Let        $$V_i=\{Z_{i1},\dots,Z_{im_i}\}\mbox{ for }1\leq i\leq d.$$
        Then, according to the $f_i(t)$-type MPL property of weight $w_i$ with respect to $\alpha$, there exists        $$\Phi_{ij}(t)\in\Mat_{d_{ij}\times d_{ij}}(\overline{k}[t])\mbox{ and }\psi_{ij}(t)\in\Mat_{d_{ij}\times 1}(\mathcal{E}_v)$$
        associated from $Z_{ij}$ satisfying Definition~\ref{Def:MZ}. Now, we consider the matrix        $$\Phi:=\oplus_{i=1}^d(\oplus_{j=1}^{m_j}(1-\alpha t)^{w_1-w_i}\Phi_{ij})$$ 
        and the vector        $$\psi:=\oplus_{i=1}^d(\oplus_{j=1}^{m_j}\Omega_\alpha^{w_1-w_i}\psi_{ij}).$$ Then, we have $\psi^{(-1)}=\Phi\psi$. Moreover, $\psi(\alpha^{-1})$ is of the form
        $$\psi(\alpha^{-1})=\oplus_{j=1}^{m_1}(\Tilde{\pi}_\alpha^{w},\star,\dots,\star,c_{ij}Z_{ij}\Tilde{\pi}_\alpha^{w})^{\tr}\oplus_{i=2}^{d}\left(\oplus_{j=1}^{m_i}(\Tilde{\pi}_\alpha^{w},\star,\dots,\star,c_{ij}Z_{ij}\Tilde{\pi}_\alpha^{w})^{\tr}\right)$$
        and $\psi(\alpha^{-q^N})$ is of the form
        $$\psi(\alpha^{-q^N})=\oplus_{j=1}^{m_1}(0,\dots,0,(c_{ij}Z_{ij}\Tilde{\pi}_\alpha^{w})^{q^N})^{\tr}\oplus_{i=2}^{d}\left(\oplus_{j=1}^{m_i}(0,\dots,0)^{\tr}\right)$$
        for all positive integers $N$. Here, the last equality comes from the fact that $\Omega_\alpha$ has simple zero at $\alpha^{-q^N}$ for all positive integers $N$ and $w_1>w_i$ for each $2\leq i\leq d$. Let 
        $$\rho:=(\rho_{11},\dots,\rho_{1m_1},\dots,\rho_{d1},\dots,\rho_{dm_d})$$
        be a nonzero vector with entries in $\overline{k}$ such that $$\rho\psi(\alpha^{-1})=0\mbox{ where }\rho_{ij}\in\Mat_{1\times d_{ij}}(\overline{k}).$$
        Since we assume that there is a nontrivial $\overline{k}$-linear relation between $$V_1\mbox{ and }\{1\}\cup\bigcup_{i=2}^dV_i,$$ the last entry of $\rho_{1\ell}$ can be chosen to be nonzero for some $1\leq \ell\leq m_1$. Now, we apply Theorem~\ref{Thm:vABP} for each $1\leq i\leq d$. Then, there exists a nonzero vector        $$P:=(P_{11},\dots,P_{1m_1},\dots,P_{d1},\dots,P_{dm_d}),$$
        such that        $$P(\alpha^{-1})=\rho\mbox{ and }P\psi=0$$ 
        where $P_{ij}\in\Mat_{1\times d_{ij}}(\overline{k}[t])$. Note that $P_{1\ell}$ is not identically zero since $\rho_{1\ell}\neq 0$. Now, we choose a positive integer $N$ large enough so that $P_{1\ell}(\alpha^{-q^N})\neq 0$. Then, we obtain a nontrivial $\overline{k}$-linear relation among $Z_{11}^{q^N},\dots,Z_{1m_1}^{q^N}$ when we specialize at $t=\alpha^{-q^N}$. After taking the $q^N$-th root, we obtain a nontrivial $\overline{k}$-linear relation among $Z_{11},\dots,Z_{1m_1}$. Hence, $V_1$ is a linearly dependent set over $\overline{k}$, and the desired result now follows.
    \end{proof}
    
{\noindent\textit{Proof of Theorem~\ref{Thm:GMZ}}.}
    For the first part, suppose on the contrary that
    $$a_0+a_1g_{w_1}+\cdots+a_dg_{w_d}=0$$
    for some $a_i\in\overline{k}$ not all zero and  $g_{w_i}\in\overline{G}_{w_i}$ not all zero. Let $U_{w_i}\subset V_{w_i}$ be a finite set of $F_\alpha$-linear independent elements in $V_{w_i}$ such that $g_{w_i}$ is contained in the $\overline{k}$-vector space spanned by elements in $U_{w_i}$. Lemma \ref{Lem:PreGMZ2} implies that there exists $1\leq i\leq d$ such that $U_{w_i}$ is a linearly dependent set over $\overline{k}$. Lemma \ref{Lem:PreGMZ1} shows that $U_{w_i}$ is a linearly dependent set over $F_\alpha$, which leads to a contradiction. The second part is an immediate consequence of the first part together with the assumption $G_{w_1}G_{w_2}\subset G_{w_1+w_2}$ for all $w_1\geq 1$ and $w_2\geq 1$. Hence, we complete the proof of Theorem \ref{Thm:GMZ}. \qed
        
\subsection{Carlitz multiple polylogarithms}

    Recall the definition of CMPLs given in (\ref{Def:CMPL}) and the $v$-adic convergence domain $\DDConv_{\fs,v}$ given in (\ref{Def:Conv._Domain_v}). In order to prove Theorem~\ref{Thm:CMPLs}, we construct a deformation series for $v$-adic CMPLs. From now on, we assume that $\epsilon_v=1$ and thus $q_v=q$. Moreover, $\varpi_v=\theta+\lambda_v$ with $\lambda_v\in\mathbb{F}_q$. The main purpose in this section is to prove Theorem~\ref{Thm:CMPLs}.
        
    \begin{definition}
        Let $\fs=(s_1,\dots,s_r)\in\mathbb{Z}_{>0}^{r}$ be an index and $\bu=(u_1,\dots,u_r)\in (\overline{k}^\times)^r\cap\DDConv_{\fs,v}$. Then we define
        \begin{align*}
            \mathfrak{L}_{\fs,\bu}(t):=&\sum_{i_1>\cdots>i_r\geq 0}(\Omega_{\varpi_v}^{s_1}u_1)^{(i_1)}\dots(\Omega_{\varpi_v}^{s_r}u_r)^{(i_r)}t^{i_1 s_1+\cdots+i_r s_r}\\
            =\Omega^{\wt(\fs)}_{\varpi_v}&\sum_{i_1>\cdots>i_r\geq 0}\frac{u_1^{q^{i_1}}\cdots u_r^{q^{i_r}}t^{i_1 s_1+\cdots+i_r s_r}}{((1-\varpi_v^{q}t)\cdots(1-\varpi_v^{q^{i_1}}t))^{s_1}\cdots((1-\varpi_v^{q}t)\cdots(1-\varpi_v^{q^{i_r}}t))^{s_r}}.
        \end{align*}
    \end{definition}
    Note that for $1\leq\ell\leq r$, since $\|\Omega_{\varpi_v}\|_v=1$ and $\bu\in(\overline{k}^\times)^r\cap\DDConv_{\fs,v}$, we have
    \[
        \|(\Omega_{\varpi_v}^{s_1}u_1)^{(i_1)}\dots(\Omega_{\varpi_v}^{s_r}u_r)^{(i_r)}t^{i_1 s_1+\cdots+i_r s_r}\|_v=|u_1^{q^{i_1}}\cdots u_r^{q^{i_r}}|_v\to 0~\mathrm{as}~0\leq i_r<\cdots<i_1\to\infty.
    \]
    Thus, we have $\mathfrak{L}_{\fs,\bu}\in\mathbb{T}_v$. The following Proposition shows that $\mathfrak{L}_{\fs,\bu}$ in fact lies in the ring of $v$-adic entire series with algebraic coefficients.
        
    \begin{proposition}\label{pMZ_1}
        Let $\fs=(s_1,\dots,s_r)\in\mathbb{Z}_{>0}^{r}$ be an index and $\bu=(u_1,\dots,u_r)\in (\overline{k}^\times)^r\cap\DDConv_{\fs,v}$. Then we have
        \[
            \mathfrak{L}_{\fs,\bu}(t)\in\mathcal{E}_v.
        \]
    \end{proposition}
        
    \begin{proof}
        We first check the convergence. By definition we have
        \begin{align*}
            \mathfrak{L}_{\fs,\bu}(t)
            &=\sum_{i_1>\cdots>i_r\geq 0}\left(\frac{\Omega_{\varpi_v}t^{i_1}}{(1-\varpi_v^qt)\cdots(1-\varpi_v^{q^i_1}t)}\right)^{s_1}\cdots\left(\frac{\Omega_{\varpi_v}t^{i_r}}{(1-\varpi_v^qt)\cdots(1-\varpi_v^{q^i_r}t)}\right)^{s_r}u_1^{q^{i_1}}\cdots u_r^{q^{i_r}}\\
            &=\sum_{i_1>\cdots>i_r\geq 0}\left(t^{i_1}\prod_{j_1>i_1}(1-\varpi_v^{q^{j_1}}t)\right)^{s_1}\cdots\left(t^{i_r}\prod_{j_r>i_r}(1-\varpi_v^{q^{j_r}}t)\right)^{s_r}u_1^{q^{i_1}}\cdots u_r^{q^{i_r}}.
        \end{align*}
        If we express 
        $$\left(\prod_{j_1>i_1}(1-\varpi_v^{q^{j_1}}t)\right)^{s_1}\cdots\left(\prod_{j_r>i_r}(1-\varpi_v^{q^{j_r}}t)\right)^{s_r}=\sum_{\ell\geq 0}c_\ell t^\ell,$$
        then we have $$|c_\ell|_vR^\ell\to 0\mbox{ as }\ell\to\infty\mbox{ for all }R>0.$$ Thus, it suffices to determine the radius of convergence of the series 
        $$\sum_{i_1>\cdots>i_r\geq 0,~\ell\geq 0}
        c_\ell u_1^{q^{i_1}}\cdots u_r^{q^{i_r}}t^{i_1s_1+\cdots+i_rs_r+\ell}.$$
        The condition $i_1>\cdots>i_r\geq 0$ and $|u_1|_v<1$ implies that 
        $$|u_1^{q^{i_1}}\cdots u_r^{q^{i_r}}|_vR^{i_1s_1+\cdots+i_rs_r}\leq |u_1|_v^{q^{i_1}}(R^{\wt(\fs)})^{i_1},\mbox{ if }R>1,$$ 
        and 
        $$|u_1^{q^{i_1}}\cdots u_r^{q^{i_r}}|_vR^{i_1s_1+\cdots+i_rs_r}\leq |u_1|_v^{q^{i_1}}R^{i_1},\mbox{ if }0<R\leq 1.$$ 
        But in any case we can conclude that 
        $$|c_\ell u_1^{q^{i_1}}\cdots u_r^{q^{i_r}}|_vR^{i_1s_1+\cdots+i_rs_r+\ell}\to 0\mbox{ as }i_1s_1+\cdots i_rs_r+\ell\to \infty\mbox{ for all }R>0.$$
        In other words, $\mathfrak{L}_{\fs,\bu}$ defines an entire function on $\mathbb{C}_v$. We now check the algebracity for the coefficients. We will prove it by using induction on the depth $r$. For $r=1$, we have
        \begin{equation}\label{functional_eq_1}
            \mathfrak{L}_{\fs,\bu}^{(-1)}=u_1^{q^{-1}}(1-\varpi_vt)^{s_1}\Omega_{\varpi_v}^{s_1}+t^{s_1}\mathfrak{L}_{\fs,\bu}.
        \end{equation}
        Since $\Omega_{\varpi_v}^{s_1}\in\mathcal{E}_v$ by Proposition~\ref{Prop:Basic_Ex}, we have $u_1^{q^{-1}}(1-\varpi_vt)^{s_1}\Omega_{\varpi_v}^{s_1}\in\mathcal{E}_v$. Thus we can express
        \[
            u_1^{q^{-1}}(1-\varpi_vt)^{s_1}\Omega_{\varpi_v}^{s_1}=\sum_{i\geq 0}d_it^i,
        \]
        where $\{d_i\}\subset \overline{k}$ and $[k_v(d_0,d_1,\dots):k_v]<\infty$. Now we express $\mathfrak{L}_{\fs,\bu}=\sum_{i\geq 0}C_it^i$ for some $C_i\in\mathbb{C}_v$. By
        (\ref{functional_eq_1}), we have
        \[
            C_\ell^{1/q}=d_\ell\mbox{, if }\ell<s_1
        \]
        and
        \[
            C_\ell^{1/q}=d_\ell+C_{\ell-s_1}\mbox{, if }\ell\geq s_1.
        \]
        This implies the algebracity for the coefficients of $\mathfrak{L}_{\fs,\bu}$ and we complete the case of $r=1$. Now we assume that the desired property $\mathfrak{L}_{\fs,\bu}\in\mathcal{E}_v$ holds for all $\fs$ with $r=\dep(\fs)\leq\ell-1$ and $\bu\in(\overline{k}^\times)^r\cap\DDConv_{\fs,v}$. To establish the case of $r=\ell$, we note that
        \begin{equation}\label{functional_eq_2}
            \mathfrak{L}_{\fs,\bu}^{(-1)}=u_\ell^{q^{-1}}(1-\varpi_vt)^{s_\ell}\Omega_{\varpi_v}^{s_\ell}\mathfrak{L}_{\fs',\bu'}+t^{s_1+\cdots+s_\ell}\mathfrak{L}_{\fs,\bu}
        \end{equation}
        where $\fs'$ (resp. $\bu'$) is the sub-index (resp. sub-tuple) obtaining from $\fs$ (resp. $\bu$) by deleting the last entry. Then the induction hypothesis and the same argument in the case $r=1$ imply the desired result.
    \end{proof}
    
    In what follows, we study the special values of $\mathfrak{L}_{\fs,\bu}$. In particular, CMPLs appear as specializations of $\mathfrak{L}_{\fs,\bu}$ at the specific algebraic points.
    
    \begin{proposition}\label{pMZ_2}
        Let $\fs=(s_1,\dots,s_r)\in\mathbb{Z}_{>0}^{r}$ be an index and $\bu=(u_1,\dots,u_r)\in (\overline{k}^\times)^r\cap\DDConv_{\fs,v}$. For any non-negative integer $N$, we have
        \[
            \mathfrak{L}_{\fs,\bu}(\varpi_v^{-q^N})=\left(\Tilde{\pi}_{\varpi_v}^{s_1+\cdots+s_r}\Li_\fs(u_1,\dots,u_r)_v\right)^{q^N}.
        \]
    \end{proposition}
        
    \begin{proof}
        We begin with the case of $N=0$. Consider the general term of $\mathfrak{L}_{\fs,\bu}(\varpi_v^{-1})$.
        \begin{align*}
            \frac{u_1^{q^{i_1}}\cdots u_r^{q^{i_r}}\varpi_v^{-(i_1 s_1+\cdots+i_r s_r)}}{\prod_{1\leq\ell\leq r}\bigg((1-\varpi_v^{q-1})\cdots(1-\varpi_v^{q^{i_\ell}-1})\bigg)^{s_\ell}}
            &=\frac{u_1^{q^{i_1}}\cdots u_r^{q^{i_r}}}{\prod_{1\leq\ell\leq r}\bigg((\varpi-\varpi_v^{q})\cdots(\varpi-\varpi_v^{q^{i_\ell}})\bigg)^{s_\ell}}\\
            &=\frac{u_1^{q^{i_1}}\cdots u_r^{q^{i_r}}}{\prod_{1\leq\ell\leq r}\bigg((\theta-\theta^{q})\cdots(\theta-\theta^{q^{i_\ell}})\bigg)^{s_\ell}}.
        \end{align*}
        The second equality comes from the fact that $\varpi_v$ is of the form $\theta+\lambda_v$ with $\lambda_v\in\mathbb{F}_q$ and $\lambda_v$ is fixed by $q$-th power map. The case $N=0$ follows immediately.
        Now we consider the case of $N>0$. By definition we have 
        \[
            \left(\Omega_{\varpi_v}^{(i_\ell)}t^{i_\ell}\right)^{s_\ell}=\left(\frac{\Omega_{\varpi_v}t^{i_\ell}}{(1-\varpi_v^qt)\cdots(1-\varpi_v^{q^{i_\ell}}t)}\right)^{s_\ell}=\left(t^{i_\ell}\prod_{j_\ell>i_\ell}(1-\varpi_v^{q^{j_\ell}}t)\right)^{s_\ell}.
        \]
        In particular, if $i_\ell<N$ then 
        \[
            \left(\Omega_{\varpi_v}^{(i_\ell)}t^{i_\ell}\right)^{s_\ell}\mid_{t=\varpi_v^{-q^N}}=\left(t^{i_\ell}\prod_{j_\ell>i_\ell}(1-\varpi_v^{q^{j_\ell}}t)\right)^{s_\ell}\mid_{t=\varpi_v^{-q^N}}=0.
        \]
        Consequently, we have 
        \begin{align*}
            \mathfrak{L}_{\fs,\bu}(\varpi_v^{-q^N})&=\sum_{i_1>\cdots>i_r\geq 0}(\Omega_{\varpi_v}^{s_1}u_1)^{(i_1)}\dots(\Omega_{\varpi_v}^{s_r}u_r)^{(i_r)}t^{i_1 s_1+\cdots+i_r s_r}\mid_{t=\varpi_v^{-q^N}}\\
            &=\sum_{i_1>\cdots>i_r\geq 0}\left(\Omega_{\varpi_v}^{(i_1)}t^{i_1}\right)^{s_1}\cdots\left(\Omega_{\varpi_v}^{(i_1)}t^{i_1}\right)^{s_1}u_1^{q^{i_1}}\cdots u_r^{q^{i_r}}\mid_{t=\varpi_v^{-q^N}}\\
            &=\sum_{i_1>\cdots>i_r\geq N}\left(\Omega_{\varpi_v}^{(i_1)}t^{i_1}\right)^{s_1}\cdots\left(\Omega_{\varpi_v}^{(i_1)}t^{i_1}\right)^{s_1}u_1^{q^{i_1}}\cdots u_r^{q^{i_r}}\mid_{t=\varpi_v^{-q^N}}\\
            &=\left(\sum_{i_1>\cdots>i_r\geq 0}(\Omega_{\varpi_v}^{s_1}u_1)^{(i_1)}\dots(\Omega_{\varpi_v}^{s_r}u_r)^{(i_r)}t^{i_1 s_1+\cdots+i_r s_r}\mid_{t=\varpi_v^{-1}}\right)^{q^N}\\
            &=\left(\Tilde{\pi}_{\varpi_v}^{s_1+\cdots+s_r}\Li_\fs(u_1,\dots,u_r)_v\right)^{q^N}.
        \end{align*}
    \end{proof}
        
    {\textbf{Proof of Theorem~\ref{Thm:CMPLs}}:}
    Recall our notation that $\fs=(s_1,\dots,s_r)\in\mathbb{Z}_{>0}^{r}$ is an index and $\bu=(u_1,\dots,u_r)\in (\overline{k}^\times)^r\cap\DDConv_{\fs,v}$. We define
    $\fs_\ell:=(s_1,\dots,s_\ell),~\bu_\ell:=(u_1,\dots,u_\ell)$ and 
    $$\psi=
        \left(
        \begin{array}{clr}
            \Omega_{\varpi_v}^{s_1+\cdots+s_r}  \\
            \Omega_{\varpi_v}^{s_2+\cdots+s_r}\mathcal{L}_{\fs_1,\bu_1} \\
            \vdots  \\
            \Omega_{\varpi_v}^{s_r}\mathcal{L}_{\fs_{r-1},\bu_{r-1}}\\
            \mathcal{L}_{\fs_{r},\bu_{r}}
        \end{array}
    \right)\in\Mat_{\ell\times 1}(\mathcal{E}_v).
    $$
    Consider the matrix $\Phi\in\Mat_{\ell\times\ell}(\overline{k}[t])$ which is defined by
    $$
    \begin{pmatrix}
    (1-\varpi_vt)^{s_1+\cdots+s_r} & 0 & 0 & \cdots & 0 \\
    u_1^{q^{-1}}(1-\varpi_vt)^{s_1+\cdots+s_r} & t^{s_1}(1-\varpi_vt)^{s_2+\cdots+s_r} & 0 & \cdots & 0\\
    0 & u_2^{q^{-1}}t^{s_1}(1-\varpi_vt)^{s_2+\cdots+s_r} & \ddots &  & \vdots \\
    \vdots &  & \ddots & t^{s_1+\cdots+s_{r-1}}(1-\varpi_vt)^{s_r} & 0\\
    0 & \cdots & 0 & u_r^{q^{-1}}t^{s_1+\cdots+s_{r-1}}(1-\varpi_vt)^{s_r} & t^{s_1+\cdots+s_r}
    \end{pmatrix}.
    $$
    Then one can check directly that $\psi^{(-1)}=\Phi\psi$. For $w\geq 1$, let 
    $$V_w:=\{\Li_\fs(\bu)_v\mid r\in\mathbb{Z}_{>0},~\fs\in\mathbb{Z}_{>0}^r,~\bu\in (\overline{k}^\times)^r\cap\DDConv_{\fs,v},~\wt(\fs)=w\}.$$
    By Proposition~\ref{pMZ_1} and Proposition~\ref{pMZ_2}, we deduce that $(\mathcal{L}_v,\{\mathcal{L}_{v,w}\},\{V_w\},\{t^w\},\varpi_v)$ forms a $v$-adic graded MPL system. Then Theorem \ref{Thm:GMZ} gives the desired result.\\
    \qed
    
\subsection{Multiple zeta values}
    For any index $(s_1,\dots,s_r)\in\mathbb{Z}_{>0}^r$, the $\infty$-adic MZV $\zeta_A(\fs)$ is defined in (\ref{Eq:MZV}). To introduce the definition of $v$-adic MZV, we define the $\fs$-th Carlitz multiple star polylogarithm, abbreviated as CMSPLs, as follows (see~\cite{CM19}):
    \begin{equation}\label{Def:CMSPL}
        \Li_\fs^\star(z_1,\dots,z_r):=\underset{i_1\geq\cdots\geq i_r\geq 0}{\sum}\frac{z_1^{q^{i_1}}\cdots z_r^{q^{i_r}}}{L_{i_1}^{s_1}\cdots L_{i_r}^{s_r}}\in k\llbracket z_1,\cdots,z_r \rrbracket.
    \end{equation}
    We denote by $\Li_\fs^\star(z_1,\cdots,z_r)_v$ when we regard this power series as a $v$-adic function. Note that $\Li_\fs^\star$ also converges on
    $$\DDConv_{\fs,v}=\{(z_1,\dots,z_r)\in\mathbb{C}_v^r\mid |z_1|_v<1,|z_\ell|_v\leq 1\mbox{ for }2\leq\ell\leq r\}.$$
    
    In what follows, we recall an important identity given by Chang and Mishiba in \cite{CM21} which is based on the result established by Chang in \cite{Cha14}.
    
    \begin{theorem}[{\cite[Theorem~5.2.5]{CM21}}, cf. {\cite[Theorem~5.5.2]{Cha14}}]\label{Thm:MZVs_to_CMSPLs}
        For any index $\fs=(s_1,\dots,s_r)\in\mathbb{Z}_{>0}^r$, we have
        \begin{equation}\label{Eq:MZVs_to_CMSPLs}
            \zeta_A(\fs)=\sum_\ell b_\ell\cdot\Li^\star_{\fs_\ell}(\bu_\ell)\in k_\infty
        \end{equation}
        for some explicit $\fs_\ell\in\mathbb{Z}_{>0}^{\dep(\fs_\ell)}$ with $\wt(\fs_\ell)=\wt(\fs)$ and $\dep(\fs_\ell)\leq\dep(\fs)$, $\bu_\ell\in A^{\dep(\fs_\ell)}$ and $b_\ell\in k$.
    \end{theorem}
    
    Let $\mathfrak{p}\in\mathbb{Z}$ be a prime number. Inspired by Furusho's strategy for defining $\mathfrak{p}$-adic MZVs, Chang and Mishiba \cite{CM19} extend the defining domain of $\Li_\fs^\star$ to the set
    $$\DDDef_{\fs,v}=\{(z_1,\dots,z_r)\in\mathbb{C}_v^r\mid |z_\ell|_v\leq 1\mbox{ for }1\leq\ell\leq r\}.$$
    In particular, we have $A^r\subset\DDDef_{\fs,v}$. Then the $v$-adic MZVs defined by Chang and Mishiba in \cite{CM21} are given by the following.
    
    \begin{definition}[{\cite[Def.~6.1.1]{CM21}}]\label{Def:vMZV}
        For any index $\fs=(s_1,\dots,s_r)\in\mathbb{Z}_{>0}^r$,
        $$\zeta_A(\fs)_v:=\sum_\ell b_\ell\cdot\Li^\star_{\fs_\ell}(\bu_\ell)_v\in k_v$$ where $\fs_\ell$, $\bu_\ell$ and $b_\ell\in k$ are given in Theorem \ref{Thm:MZVs_to_CMSPLs}.
    \end{definition}
    
    To prove Theorem \ref{Main_Thm}, we need to investigate some properties of $v$-adic CMSPLs. Let $\mathcal{L}_v^\star$ be the $k$-vector space spanned by $1$ and $\Li_\fs^\star(\bu)_v$ with $\bu\in(\overline{k}^\times)^r\cap\DDConv_{\fs,v}$. For $w\geq 1$, let $\mathcal{L}_{v,w}^\star$ be the $k$-vector space spanned by $\Li_\fs^\star(\bu)_v$ with $\bu\in(\overline{k}^\times)^r\cap\DDConv_{\fs,v}$ and $\wt(\fs)=w$. It is known that $\mathcal{L}_v^\star$ forms a $k$-algebra and $\mathcal{L}_{v,w_1}^\star\mathcal{L}_{v,w_2}^\star\subset\mathcal{L}_{v,w_1+w_2}^\star$ by the stuffle relations (see \cite[Prop.~2.2.3]{CCM22}).
    
    Let $w\in\mathbb{Z}_{>0}.$ By using inclusive-exclusive principle on the set $\{i_1\geq\cdots\geq i_r\geq 0\}$ and the definition of $\DDConv_{\fs,v}$, it is not hard to see that $v$-adic CMSPLs of weight $w$ can be written as $\mathbb{F}_q$-linear combinations of $v$-adic CMPLs of weight $w$. For example, $$\{i_1\geq i_2\geq 0\}=\{i_1>i_2\geq 0\}\cup \{i_1=i_2\geq 0\}.$$ It follows that $$\Li^\star_{(s_1,s_2)}(u_1,u_2)_v=\Li_{(s_1,s_2)}(u_1,u_2)+\Li_{(s_1+s_2)}(u_1\cdot u_2).$$
    
    Conversely, $v$-adic CMPLs of weight $w$ can be written as $\mathbb{F}_q$-linear combinations of $v$-adic CMSPLs of weight $w$ by similar argument. (cf. \cite[Prop.~5.2.3]{CM21}). Thus, the following proposition is straightforward.
    
    \begin{proposition}\label{Prop:CMSPLs=CMPLs}
        Let $w\in\mathbb{Z}_{>0}$. Then the following assertion holds.
        $$\mathcal{L}_{v,w}^\star=\mathcal{L}_{v,w}.$$
    \end{proposition}

    By using Proposition~\ref{Prop:CMSPLs=CMPLs} together with Theorem~\ref{Thm:CMPLs}, we are able to give the proof of Theorem~\ref{Main_Thm}.
    
    {\textbf{Proof of Theorem~\ref{Main_Thm}}:~}
    Let $\fs\in\mathbb{Z}_{>0}^r$ be any index and $w=\wt(\fs)\in\mathbb{Z}_{>0}$. It is known from \cite[Cor.~3.2.11]{Che23} (see also \cite[Rem.5.2.1~]{CCM22}) that $\zeta_A(\fs)_v\in\mathcal{L}_{v,w}^\star$. Then Proposition \ref{Prop:CMSPLs=CMPLs} implies that $\zeta_A(\fs)_v\in\mathcal{L}_{v,w}$ and thus $\mathcal{Z}_{v,w}\subset\mathcal{L}_{v,w}$.
    The desired result now follows from Theorem~\ref{Thm:CMPLs} immediately.\\
    \qed

    \begin{remark}
        Recall that $\overline{\mathcal{Z}}$ is the $\overline{k}$-algebra generated by all $\infty$-adic MZVs. It is shown in \cite[Thm.~1.2.3]{CCM22} that the natural map $\phi_v:=(\zeta_A(\fs)\mapsto\zeta_A(\fs)_v):\overline{\mathcal{Z}}\to\overline{\mathcal{Z}}_v$ is a well-defined surjective $\overline{k}$-algebra homomorphism. Moreover, they assert a conjecture \cite[Conj.~5.4.1]{CCM22} that the kernel of $\phi_v$ is the principal ideal generated by $\zeta_A(q-1)$. As an immediate consequence of Theorem \ref{Main_Thm}, we provide some information for the kernel of $\phi_v$. More precisely, Theorem~\ref{Main_Thm} implies that if $\epsilon_v=1$, then $\Ker\phi_v\subset\overline{\mathcal{Z}}$ is a homogeneous ideal.
    \end{remark}
    
    
    
\appendix
\section{}
\subsection{Some Lemmas}
    Throughout this section, let $k_0/k$ be a finite separable extension and $k_1$ be the perfection of $k_0$ in $\overline{k}$. Let $R_v\subset k$ be the ring of rational functions regular away from $v$. Let $R_1$ be the integral closure of $R_v$ in $k_1$ and $R_0:=k_0\cap R_1$. Note that $\overline{k}$ is the union of all its subfields of the form $k_1$. 
    For the convenience of our later use, we put 
    $$||x||_v:=\max_{\tau\in\Aut(\overline{k}/k)}|\tau x|_v\mbox{ for all }x\in\overline{k},$$
    $$||f||_v:=\max_{i=0}^r||a_i||_v\mbox{ for all }f=\sum_{i=0}^ra_it^i\in\overline{k}[t],$$
    and 
    $$||M||_v:=\max_{i,j}||M_{ij}||_v\mbox{ for all } M\in\Mat_{\ell\times\ell}(\overline{k}[t]).$$ 
    We further set 
    $$\deg_t M:=\max_{i,j}(\deg_tM_{ij})\mbox{ for all } M\in\Mat_{\ell\times\ell}(\overline{k}[t]).$$

To prove Theorem \ref{Thm:vABP}, we require some lemmas.
        
    \begin{lemma}[cf. {\cite[~(2.5)]{ABP04}}]\label{Schwarz_Jensen_formula}
        Fix $0\neq f\in\mathcal{E}_v$. Let $\{\omega_i\}$ be an enumeration of zeros of $f$ in $\mathbb{C}_v$ and $\lambda$ be the constant term of $f$. Then, we have        $$\sup_{x\in\mathbb{C}_v,|x|_v\leq r}|f(x)|_v=|\lambda|_v\cdot r^{\#\{i\mid\omega_i=0\}}\prod_{i:0<|\omega_i|_v<r}\frac{r}{|\omega_i|_v},~r\in\mathbb{R}_{>0}.$$ 
        In particular, if $r>1$ and$\{\omega_{i_j}\}_{j=1}^{n}\subset \{\omega_i\}$ is a finite subset of zeros of $f$ in $\mathbb{C}_v$, then 
        $$\sup_{x\in\mathbb{C}_v,|x|_v\leq r}|f(x)|_v\geq |\lambda|_v\cdot\prod_{j=1}^n\frac{r}{|\omega_{i_j}|_v}.$$
    \end{lemma}
        
    \begin{proof}
        The desired result follows immediately by an analogue of the classical Weierstrass factorization theorem \cite[Thm.~2.14]{Gos96}.
    \end{proof}
        
    \begin{lemma}[cf. {\cite[~(3.3.2)]{ABP04}}]\label{lower_bound_from_size}
        Let $0\neq x\in R_1$. Then, we have 
        $$||x||_v\geq 1,~|x|_v\geq ||x||_v^{1-[k_0:k]}.$$
    \end{lemma}
        
    \begin{proof}
        Let $0\neq x\in R_0$. Then 
        $$||x||_v=(||x||_v^{[k_0:k]})^{1/[k_0:k]}\geq(\prod_{\tau\in\Aut(k_0/k)}|\tau x|_v)^{1/[k_0:k]}=|N_{k_0/k}(x)|_v^{1/[k_0:k]}\geq 1.$$ 
        The last inequality comes from the fact that $N_{k_0/k}(x)\in R_0\cap k=R_v$. However, we have
        $$|x|_v||x||_v^{[k_0:k]-1}\geq \prod_{\tau\in\Aut(k_0/k)}|\tau x|_v=|N_{k_0/k}(x)|_v\geq 1.$$ 
        Thus, the desired result holds for $0\neq x\in R_0$. However, note that 
        $$R_1=\bigcup_{\mu=0}^{\infty}R_0^{q^{-\mu}}.$$ 
        Hence, the lemma follows immediately.
    \end{proof}
    
    \begin{lemma}[cf. {\cite[Lem.~3.3.3]{ABP04}}]\label{Liouville_inequality}
        Let $$f(z):=\sum_{i=0}^na_iz^i\in R_1[z],$$ such that $f(z)$ does not vanish identically. For every nonzero root $\lambda\in\overline{k}$ of $f(z)$ of order $\mu$, we have $$|\lambda|_v^\mu\geq (\max_{i=0}^n||a_i||_v)^{-[k_0:k]}.$$
    \end{lemma}
    
    \begin{proof}
        Since we have the inequality $(\max_{i=0}^n||a_i||_v)^{-[k_0:k]}\leq 1$,
        the statement is valid in the case of $|\lambda|_v\geq 1$.
        Thus, we may assume $|\lambda|_v<1$. After factoring out a power of $z$, we may further assume $a_0\neq 0$. Let        $$f(z+\lambda)=\sum_{j=\mu}^nb_jz^j.$$ 
        Note that 
        $$|b_j|_v\leq \max_{i=j}^n|a_i|_v\leq \max_{i=0}^n||a_i||_v.$$ 
        Thus, 
        $$|a_0|_v=|f(0)|_v=|f(z+\lambda)\mid_{z=-\lambda}|_v\leq \max_{i=\mu}^n|b_i\lambda^i|_v\leq|\lambda|_v^\mu\max_{i=0}^n||a_i||_v.$$ 
        By Lemma \ref{lower_bound_from_size}, we have 
        $$|a_0|_v\geq ||a_0||_v^{1-[k_0:k]}\geq (\max_{i=0}^n||a_i||_v)^{1-[k_0:k]}.$$ 
        The desired result now follows immediately.
    \end{proof}
    
    The normalization $|\varpi_v|_v=q_v$ was imposed to make the following Lemma hold.
    
    \begin{lemma}[cf. {\cite[Lem.~3.3.4]{ABP04}}]\label{counting_lemma}
        For all constants $C>1$, we have        $$\lim_{\mu\to\infty}(\#\{x\in R_0^{q^{-\mu}}\mid||x||_v\leq C\})^{\frac{1}{q^\mu[k_0:k]}}=C.$$
    \end{lemma}
        
    \begin{proof}
        Since 
        $$\{q_v^\delta\mid\delta\in\bigcup_{\mu=0}^\infty q^{-\mu}\mathbb{Z},\delta>0\}\subset\mathbb{R}_{>1}$$ 
        is dense in $\mathbb{R}_{>1}$, we may assume        $$C=q_v^\delta~~(\delta\in\bigcup_{\mu=0}^\infty q^{-\mu}\mathbb{Z},\delta>0).$$ 
        
        Let $X$ be the smooth projective curve whose closed points are the places of $k_0$. Let $\mathbb{F}_0$ be the constant field of $k_0$. Suppose that $v_1,\dots,v_r$ are places of $k_0$ lying above $v$. Let $\varpi_{v_i}$ be the corresponding prime ideal of the place $v_i$ and $\mathbb{F}_{v_i}$ be the residue field of $v_i$. Then, $\varpi_{v_1},\dots,\varpi_{v_r}$ are prime ideals lying above $\varpi_v$ in $k_0$. Suppose that $\varpi_vR_0=\varpi_{v_1}^{e_1}\cdots\varpi_{v_r}^{e_r}$.
        Consider the effective divisor of $X$
        $$D_n:=n(e_1v_1+\cdots+e_rv_r)\in\mathrm{Div}(X).$$ 
        Then
        \begin{align*}
            \deg(D_n)&=n\sum_{i=1}^re_i[\mathbb{F}_{v_i}:\mathbb{F}_0]\\
            &=n\sum_{i=1}^re_i\frac{[\mathbb{F}_{v_i}:\mathbb{F}_q]}{[\mathbb{F}_0:\mathbb{F}_q]}\\
            &=n\sum_{i=1}^re_i\frac{[\mathbb{F}_{v_i}:\mathbb{F}_v][\mathbb{F}_v:\mathbb{F}_q]}{[\mathbb{F}_0:\mathbb{F}_q]}\\
            &=\frac{n\epsilon_v[k_0:k]}{[\mathbb{F}_0:\mathbb{F}_q]}.
        \end{align*}
        Here, the second equality comes from the inclusion of fields $\mathbb{F}_q\subset\mathbb{F}_0\subset\mathbb{F}_{v_i}$, while the third equality follows by $\mathbb{F}_q\subset\mathbb{F}_v\subset\mathbb{F}_{v_i}$.
        Now, the Riemann-Roch theorem provides constants $n_0,n_1$ such that for $n>n_0$,
        \begin{align*}
            \ell(D_n)&:=\dim_{\mathbb{F}_0}\left(\{f\in k_0\mid (f)+n\sum_{i=0}^re_i\varpi_{v_i}\geq 0\}\cup\{0\}\right)\\
            &=[\mathbb{F}_0:\mathbb{F}_q]\dim_{\mathbb{F}_q}\left(\{f\in k_0\mid (f)+n\sum_{i=0}^re_i\varpi_{v_i}\geq 0\}\cup\{0\}\right)\\
            &=[\mathbb{F}_0:\mathbb{F}_q]\deg(D_n)+n_1\\
            &=n\epsilon_v[k_0:k]+n_1.
        \end{align*}
        
        We claim that 
        \[
            \{f\in k_0\mid (f)+n\sum_{i=0}^re_i\varpi_{v_i}\geq 0\}\cup\{0\}=\{x\in R_0\mid ||x||_v\leq q_v^n\}.
        \]
        Indeed, let 
        \[
            f\in\{f\in k_0\mid (f)+n\sum_{i=0}^re_i\varpi_{v_i}\geq 0\}\cup\{0\}.
        \]
        Then, we have $\mathrm{ord}_\omega(f)\geq 0$ for each place $\omega$ of $k_0$ such that $\omega$ does not divide $v$. Since  $R_v=\bigcap\mathcal{O}_{v'}$  where $\mathcal{O}_{v'}$ is the ring of functions inside $k$ which are regular at $v'$, and $v'$ runs over all places of $k$ which is not equal to $v$, by \cite[Thm.~3.2.6]{Sti93} we have $R_0=\bigcap\mathcal{O}_\omega$, where $\mathcal{O}_\omega$ is the ring of functions inside $k_0$ which are regular at $\omega$, and $\omega$ runs over all places of $k_0$, which does not divide $v$. Thus, we conclude that $f\in R_0$.        Moreover, 
        \[
            f\in\{f\in k_0\mid (f)+n\sum_{i=0}^re_i\varpi_{v_i}\geq 0\}\cup\{0\}
        \]
        implies that $\mathrm{ord}_{v_i}(f)\geq -ne_i$. In particular, by \cite[Thm.~8.1]{Neu13}, we obtain $||f||_v\leq q_v^n$ as desired. To see the other direction, let $f\in\{x\in R_0\mid ||x||_v\leq q_v^n\}$. Then, by \cite[Thm.~8.1]{Neu13} again, we have $\mathrm{ord}_{v_i}(f)\geq -ne_i$. In other words, we conclude that 
        \[
            f\in\{f\in k_0\mid (f)+n\sum_{i=0}^re_i\varpi_{v_i}\geq 0\}\cup\{0\}
        \]
        and the desired equality follows immediately.
        
        Now, for sufficiently large $n$, we have        $$\#\{x\in R_0\mid ||x||_v\leq q_v^n\}=q_v^{n[k_0:k]+n_1/\epsilon_v}.$$ 
        In particular, we have
        $$\#\{x\in R_0^{-q^\mu}\mid||x||_v\leq q_v^\delta\}=\#\{x\in R_0\mid||x||_v\leq q_v^{q^\mu\delta}\}=q_v^{q^\mu\delta[k_0:k]+n_1/\epsilon_v}.$$
        Hence, we obtain the desired equality
        $$\lim_{\mu\to\infty}(\#\{x\in R_0^{q^{-\mu}}\mid ||x||_v\leq q_v^\delta\})^{\frac{1}{q^\mu[k_0:k]}}=q_v^\delta.$$
    \end{proof}
        
    \begin{lemma}[cf. {\cite[Lem.~3.3.5]{ABP04}}]\label{Thue_Sigel_analogue}
        Fix $C>1$, $0<r<s$ $(C\in\mathbb{R},r,s\in\mathbb{Z})$. For each matrix        $$M\in\Mat_{r\times s}(R_1)$$ 
        such that        $$||M||_v:=\max_{i,j}||M_{ij}||_v<C,$$ 
        there exists $$x\in\Mat_{s\times 1}(R_1)$$ such that
        $$x\neq 0\mbox{, }Mx=0\mbox{, and }||x||_v<C^{\frac{r}{s-r}}.$$
    \end{lemma}
        
    \begin{proof}
        Choose $C'>1$ and $\epsilon>0$ such that        $$||M||_v<C',~(1+\epsilon)(C')^{\frac{r}{s-r}}<C^{\frac{r}{s-r}}.$$ 
        Let $\epsilon'>0$ be chosen small enough. Then, for a sufficiently large $\mu_0$, we have        $$|\#\{x\in\Mat_{s\times 1}(R_0^{q^{-\mu}})\mid||x||_v\leq (1+\epsilon)(C')^{\frac{r}{s-r}}\}-(1+\epsilon)^{sq^\mu[k_0:k]}(C')^{\frac{srq^\mu[k_0:k]}{(s-r)}}|<\epsilon'$$ 
        and        $$|\#\{x\in\Mat_{r\times 1}(R_0^{q^{-\mu}})\mid||x||_v\leq (1+\epsilon)(C')^{\frac{s}{s-r}}\}-(1+\epsilon)^{rq^\mu[k_0:k]}(C')^{\frac{srq^\mu[k_0:k]}{(s-r)}}|<\epsilon'$$ 
        by Lemma~\ref{counting_lemma}. Note that multiplication by $M$ maps the former set        $$\{x\in\Mat_{s\times 1}(R_0^{q^{-\mu}})\mid||x||_v\leq (1+\epsilon)(C')^{\frac{r}{s-r}}\}$$ 
        to the latter one        $$\{x\in\Mat_{r\times 1}(R_0^{q^{-\mu}})\mid||x||_v\leq (1+\epsilon)(C')^{\frac{s}{s-r}}\}.$$ 
        Finally, for a sufficiently large $\mu$, we have
        $$(1+\epsilon)^{sq^\mu[k_0:k]}(C')^{\frac{srq^\mu[k_0:k]}{(s-r)}}-\epsilon'>(1+\epsilon)^{rq^\mu[k_0:k]}(C')^{\frac{srq^\mu[k_0:k]}{(s-r)}}+\epsilon'$$ 
        and hence the existence of the desired $x$ follows by the pigeonhole principle.
    \end{proof}
        
    \begin{lemma}[cf. {\cite[Lem.~3.3.6]{ABP04}}]\label{Thue_Sigel_analogue_2}
        Fix $C>1$, $0<r<s$ $(C\in\mathbb{R},r,s\in\mathbb{Z})$. For each matrix        $$M\in\Mat_{r\times s}(R_1[t])$$ 
        such that        $$||M||_v:=\max_{i,j}||M_{ij}||_v<C,$$ 
        there exists        $$x\in\Mat_{s\times 1}(R_1[t])$$ 
        such that 
        $$x\neq 0\mbox{, }Mx=0\mbox{, and }||x||_v<C^{\frac{r}{s-r}}.$$
    \end{lemma}
        
    \begin{proof}
        For any positive integer $n$, we consider the $R_1$-basis
        $$\{(1,0,\dots,0)^\tr,\dots,(t^n,0,\dots,0)^\tr,(0,1,0,\dots,0)^\tr,\dots,(0,\dots,0,1)^\tr,\dots,(0,\dots,0,t^n)^\tr\}$$ 
        of 
        $$\{x\in\Mat_{s\times 1}(R_1[t])\mid\deg_tx\leq n\}.$$ 
        If $\deg_tM\leq d$, then multiplication by $M$ induces a $R_1$-linear mapping of the set        $$\{x\in\Mat_{s\times 1}(R_1[t])\mid\deg_tx\leq e\}$$ 
        to the set        $$\{x\in\Mat_{r\times 1}(R_1[t])\mid\deg_tx\leq d+e\}.$$
        Put $r':=r(d+e+1)$, $s':=s(e+1)$, where $e$ is chosen large enough so that        $$r'<s',~||M||_v<C':=(C)^{(\frac{r}{s-r})/(\frac{r'}{s'-r'})}.$$ 
        Using the basis we mentioned above, the $R_1$-linear map induced from multiplication by $M$ can be represented by a matrix $M'\in\Mat_{r'\times s'}(R_1)$ such that $||M'||_v<C'$. Now, the existence of $x\in\Mat_{s\times 1}(R_1[t])$ such that $x\neq 0$, $Mx=0$, $\deg_tx\leq e$, and $||x||_v<(C')^{\frac{r'}{s'-r'}}=C^{\frac{r}{s-r}}$ follows Lemma~\ref{Thue_Sigel_analogue} with parameters $(C',r',s')$.
    \end{proof}

\subsection{Proof of Theorem \ref{Thm:vABP}}
    By abuse of the notation, for $B=(B_{ij})\in\Mat_{m\times n}(\mathbb{C}_\infty\llbracket t\rrbracket)$ and $N\in\mathbb{Z}$, we denote $\sigma^NB:=(B_{ij}^{(-N)})$.
    We first consider the case $\ell=1$. Without loss of generality, we may assume $\rho\neq 0$ and thus $\psi(\gamma)=0$. It suffices to show that $\psi$ vanishes identically. For any integer $\mu\geq 0$, we have
    $$\psi(\gamma^{q^{-\mu}})^{q^{-1}}=\psi^{(-1)}(\gamma^{q^{-(\mu+1)}})=\Phi(\gamma^{q^{-(\mu+1)}})\psi(\gamma^{q^{-(\mu+1)}}).$$ 
    But
    $$\Phi(\gamma^{q^{-(\mu+1)}})\neq 0\mbox{ for all }\mu=0,1,\dots.$$ Thus, $$\psi(\gamma^{q^{-(\mu+1)}})=0\mbox{ for all }\mu=0,1,\dots.$$ Since $\psi$ has infinitely many zeros in the closed disk $$\{x\in\mathbb{C}_v\mid|x|_v\leq \max\{1,|\gamma|_v\}\},$$ we may conclude that $\psi$ vanishes identically, and thus, the case $\ell=1$ is completed. We now consider the case $\ell>1$. Without loss of generality, we may assume $$\rho\neq 0,~\gamma\in k_1,~\Phi(t)\in\Mat_{\ell\times\ell}(k_1[t]),\mbox{ and }\rho\in\Mat_{1\times\ell}(k_1).$$ For suitably chosen $a,b\in R_v$, $ab\neq 0$ and the following replacement $$a^{q-1}\Phi\to\Phi,~a^{-q}\psi\to\psi,~b\rho\to\rho,$$ we may assume $$\Phi\in\Mat_{\ell\times\ell}(R_1[t]),~\rho\in\Mat_{1\times\ell}(R_1).$$ Fix $\vartheta\in\Mat_{\ell\times(\ell-1)}(R_1)$ of maximal rank such that $$\rho\vartheta=0.$$ Then, the $k_1$-subspace of $\Mat_{1\times\ell}(k_1)$ annihilated by right multiplication of $\vartheta$ is the $k_1$-span of $\rho$. Let $N$ be a parameter taking values in positive integers, which is divided by $2\ell$. We divide our proof into the following four claims.
        \begin{claim}
            For each $N$, there exists $h_N(t)\in\Mat_{1\times\ell}(R_1[t])$ such that
            \begin{enumerate}
                \item $||h_N||_v=O(1)$ as $N\to\infty$.
                \item $h_N$ do not vanish identically.
                \item $\deg_th_N<(1-\frac{1}{2\ell})N$.
                \item For $\mu=0,1,\dots,N-1$, the function $h_N(\sigma^0\Phi^{\ad})\cdots(\sigma^{N+\mu-1}\Phi^{\ad})(\sigma^{N+\mu}\vartheta)$ vanishes at $t=\gamma^{q^{-(N+\mu)}}$.
            \end{enumerate}
        \end{claim}
            
        We first note that the hypothesis $\rho\psi(\gamma)=0$ is equivalent to $$(\sigma^{N+\mu}\rho)(\sigma^{N+\mu}\psi)(\gamma^{q^{-(N+\mu)}})=0,$$
        and then condition (4) implies that $$h_N(t)(\sigma^{0}\Phi^{\ad})\cdots(\sigma^{N+\mu-1}\Phi^{\ad})\mbox{ at }t=\gamma^{q^{-(N+\mu)}}\mbox{ for }\mu=0,1,\dots,N-1$$
        is spanned by $\sigma^{N+\mu}\rho$, and thus they are annihilated by the right multiplication of $(\sigma^{N+\mu}\psi)(\gamma^{q^{N+\mu}})$. Then
        \begin{align*}
            0&=h_N(\sigma^{0}\Phi^{\ad})\cdots(\sigma^{N+\mu-1}\Phi^{\ad})(\sigma^{N+\mu}\psi)\mid_{t=\gamma^{q^{-(N+\mu)}}}\\
            &=h_N(\sigma^{0}\Phi^{\ad})\cdots(\sigma^{N+\mu-1}\Phi^{\ad})(\sigma^{N+\mu-1}\Phi)\cdots(\sigma^{0}\Phi)\psi\mid_{t=\gamma^{q^{-(N+\mu)}}}\\
            &=[\det\Phi(\gamma^{q^{-1}})]^{q^{-(N+\mu-1)}}\cdots[\det\Phi(\gamma^{q^{-(N+\mu)}})]^{q^{-0}}E_N(\gamma^{q^{-(N+\mu)}}),
        \end{align*}
        where $E_N(t):=h_N(t)\psi(t)$. Since $\det\Phi$ does not vanish at $\sigma^i(\gamma)$ for all $i\in\mathbb{Z}_{>0}$, we can conclude that
        \begin{equation}\label{key_identity}
            E_N(\gamma^{q^{-(N+\mu)}})=0\mbox{ for all }\mu=0,1,\dots,N-1.
        \end{equation}
        Now, we can complete this claim if we can find $h_N(t)$ so that it satisfies conditions (1), (2), (3) and (4). To construct such $h_N(t)$, it suffices to determine the coefficients. We first regard $h_N(t)$ as a polynomial of degree less than $(1-\frac{1}{2\ell})N$ with indeterminate
        coefficients that satisfy a homogeneous linear system by Condition (4). Then, we can solve this linear system due to Lemma~\ref{Thue_Sigel_analogue}. Now, we set $$r:=(\ell-1)N,~s:=(\ell-\frac{1}{2})N$$ and choose $$\xi\in R_v\mbox{ such that }|\xi|_v>1\mbox{ and }\xi\gamma\in R_1.$$ For each $0\leq\mu\leq N-1$, we multiply by $$\xi^{q^{-(N+\mu)}((1-\frac{1}{2})N+2N\deg_t\Phi^{\ad})}$$ on both sides of Condition (4). Then, the homogeneous linear system arising from Condition (4) can be represented by a matrix $M_N\in\Mat_{r\times s}(R_1)$ such that $$||M_N||_v\leq (|\xi|_v|\gamma|_v)^{q^{-N}((1-\frac{1}{2\ell})N+2N\deg_t\Phi^{\ad})}||\Phi^{\ad}||_v^{\frac{q}{q-1}}||\vartheta||_v=O(1)\mbox{ as }N\to\infty\mbox{ if }|\gamma|_v>1$$ or $$||M_N||_v\leq |\xi|_v^{q^{-N}((1-\frac{1}{2\ell})N+2N\deg_t\Phi^{\ad})}||\Phi^{\ad}||_v^{\frac{q}{q-1}}||\vartheta||_v=O(1)\mbox{ as }N\to\infty\mbox{ if }|\gamma|_v\leq 1.$$ The desired coefficient of $h_N(t)$ is now described by a vector $x_N\in\Mat_{s\times 1}(R_1)$ such that $$x_N\neq 0,~M_Nx_N=0\mbox{ and }||x_N||_v=O(1)\mbox{ as }N\to\infty.$$ The existence of such $x_N$ is immediately followed by Lemma~\ref{Thue_Sigel_analogue}, and hence, we complete our claim.
                    
\begin{claim}
            For each $N$, there exists $a_{0,N}(t),\dots,a_{\ell,N}(t)\in R_1[t]$ such that
            \begin{enumerate}
                \item $\max_{i=0}^\ell||a_{i,N}||_v=O(1)$ as $N\to\infty$.
                \item Not all $a_{i,N}$ vanishes identically.
                \item $a_{0,N}(\sigma^{0}h_N)+a_{1,N}(\sigma^{1}h_N)(\sigma^{0}\Phi)+\cdots+a_{\ell,N}(\sigma^{\ell-1}\Phi)\cdots(\sigma^{0}\Phi)=0$
            \end{enumerate}
        \end{claim}
            
        We first note that Condition (3) implies that $$a_{0,N}E_N+a_{1,N}(\sigma^{1}E_N)+\cdots+a_{\ell,N}(\sigma^{\ell}E_N)=0,$$ where $E_N(t):=h_N(t)\psi(t)$. To construct such $a_{0,N},\dots,a_{\ell,N}$ satisfying conditions (1), (2) and (3), it suffices to solve a homogeneous linear system induced from Condition (3), and we can solve it due to Lemma~\ref{Thue_Sigel_analogue_2}. Indeed, we note that the homogeneous linear system can be represented by a matrix $M_N\in\Mat_{\ell\times(\ell+1)}(R_1[t])$ such that $||M_N||_v=O(1)$ as $N\to\infty$. The desired solutions $a_{0,N},\dots,a_{\ell,N}$ can be described by a vector $x_N\in\Mat_{(\ell+1)\times 1}(R_1[t])$ such that $$x_N\neq 0,~M_Nx_N=0\mbox{ and }||x_N||_v=O(1)\mbox{ as }N\to\infty.$$ The existence of such $x_N$ now immediately follows Lemma~\ref{Thue_Sigel_analogue_2}, and hence, we complete our claim. 
            
        After factoring out the common divisor, we may further assume that not all $a_{i,N}(0)=0$.
            
        \begin{claim}
            $E_N(t)$ vanishes identically for some $N$.
        \end{claim}
            
        Suppose on the contrary that $E_N(t)$ does not vanish identically for all $N$. Let $\lambda_N$ be the leading coefficient of the Taylor expansion of $E_N(t)$ at the origin. Then, we have        $$a_{0,N}(0)\lambda_N^{q^0}+\cdots+a_{\ell,N}(0)\lambda_N^{q^{-\ell}}=0.$$ 
        By Lemma~\ref{lower_bound_from_size} and the construction of $a_{0,N},\dots,a_{\ell,N}$, we have
        \begin{equation}\label{vanishing of E_N}                |\lambda_N|^{-1}_v\leq(\max_{i=0}^\ell||a_{i,N}(0)||_v)^{[k_0:k]}=O(1)\mbox{ as }N\to\infty.
        \end{equation}
        However, we also have $$|\lambda_N|_v|\gamma|_v^{N-q/(q-1)}\leq\underset{\underset{|x|_v\leq|\gamma|_v}{x\in\mathbb{C}_v}}{\sup}|E_N(x)|_v\leq(\underset{\underset{|x|_v\leq|\gamma|_v}{x\in\mathbb{C}_v}}{\sup}|\psi(x)|_v)||h_N||_v|\gamma|_v^{N(1-1/2\ell)}\mbox{ if }|\gamma|_v>1$$ or $$|\lambda_N|_v(\frac{1}{|\gamma|_v})^{N}\leq\underset{\underset{|x|_v\leq|1/\gamma|_v}{x\in\mathbb{C}_v}}{\sup}|E_N(x)|_v\leq(\underset{\underset{|x|_v\leq|1/\gamma|_v}{x\in\mathbb{C}_v}}{\sup}|\psi(x)|_v)||h_N||_v(\frac{1}{|\gamma|_v})^{N(1-1/2\ell)}\mbox{ if }|\gamma|_v<1$$ or $$|\lambda_N|_v|\alpha|_v^{N}\leq\underset{\underset{|x|_v\leq|\alpha|_v}{x\in\mathbb{C}_v}}{\sup}|E_N(x)|_v\leq(\underset{\underset{|x|_v\leq|\alpha|_v}{x\in\mathbb{C}_v}}{\sup}|\psi(x)|_v)||h_N||_v|\alpha|_v^{N(1-1/2\ell)}\mbox{ if }|\gamma|_v=1,$$ where $\alpha\in\overline{k}^\times$ is any element that $|\alpha|_v>1$. Thus, we have $$|\lambda_N|_v=O(|\gamma|_v^{-\frac{N}{2\ell}})\mbox{ as }N\to\infty\mbox{ if }|\gamma|_v>1$$ or $$|\lambda_N|_v=O((\frac{1}{|\gamma|_v})^{-\frac{N}{2\ell}})\mbox{ as }N\to\infty\mbox{ if }|\gamma|_v<1$$ or $$|\lambda_N|_v=O(|\alpha|_v^{-\frac{N}{2\ell}})\mbox{ as }N\to\infty\mbox{ if }|\gamma|_v=1.$$ However, in any case, the growth of $|\lambda_N|_v$ contradicts (\ref{vanishing of E_N}). Hence, we complete our claim.
            
        \begin{claim}
            The assertion in Theorem~\ref{Thm:vABP} holds.
        \end{claim}
            
        Now, we fix $N$ so that $E_N$ vanishes identically. Since $\deg_th_N<N$, there exists $0\leq\mu<N$ such that $(\sigma^{-(N+\mu)}h_N)(\gamma)\neq 0$. We consider $$P(t):=(\sigma^{-(N+\mu)}h_N)(\sigma^{-(N+\mu)}\Phi^{\ad})\cdots(\sigma^{-1}\Phi^{\ad})\in\Mat_{1\times\ell}(R_1[t]).$$ Since $$\det((\sigma^{-(N+\mu)}\Phi^{\ad})\cdots(\sigma^{-1}\Phi^{\ad}))\mid_{t=\gamma}\neq 0,$$ we have $P(\gamma)\neq 0$. Moreover, $$P(\gamma)\vartheta=(h_N(\sigma^{0}\Phi^{\ad})\cdots(\sigma^{(N\mu-1)}\Phi^{\ad})(\sigma^{N\mu}\vartheta)\mid_{t=\gamma^{q^{-(N\mu)}}})^{q^{N+\mu}}=0$$ implies that $$P(\gamma)\in (k_1\mbox{-span of }\rho)\subset\Mat_{1\times\ell}(k_1).$$ Finally, $$P\psi=\sigma^{-(N+\mu)}[(\sigma^{N+\mu-1}\det\Phi)\cdots(\sigma^{0}\det\Phi)E_N]=0.$$ Up to a scalar, $P(t)$ is the vector we want, and hence, we complete the proof of the theorem.\\
        \qed


\bibliographystyle{alpha}

\begin{thebibliography}{999999}




\bibitem[ABP04]{ABP04}
G. W. Anderson, W. D. Brownawell and M. A. Papanikolas, \textit{Determination of the algebraic relations among special $\Gamma$-values in positive characteristic}, Ann. of Math. (2) \textbf{160} (2004), no. 1, 237--313.


\bibitem[AT90]{AT90}
G.\ W.\ Anderson and D.\ S.\ Thakur, \textit{Tensor powers of the Carlitz module and zeta values}, Ann. of Math. (2) \textbf{132} (1990), no. 1, 159--191.





\bibitem [BGF19]{BGF19}
J.\ I.\ Burgos Gil and J.\  Fresán, \textit{Multiple zeta values: from numbers to motives}, to appear in Clay Mathematics Proceedings.


\bibitem [Car35]{Car35}
L.\ Carlitz, \textit{On certain functions connected with polynomials in a Galois field}, Duke Math. J. \textbf{1} (1935), no. 2, 137-168.

\bibitem[Cha09]{Cha09}
C.-Y.\ Chang, \textit{A note on a refined version of Anderson-Brownawell-Papanikolas criterion}, J.Number Theory. \textbf{129} (2009), 729-738.

\bibitem[Cha14]{Cha14}
C.-Y.\ Chang, \textit{Linear independence of monomials of multizeta values in positive characteristic}, Compos. Math. \textbf{150} (2014), 1789-1808.







\bibitem[CCM22]{CCM22} C.-Y. Chang, Y.-T. Chen and Y. Mishiba, {\it Algebra structure of multiple zeta values in positive characteristic}, Camb. J. Math. (10), No. 4, 743-783 (2022)



\bibitem[Che23]{Che23} Y.-T. Chen, {\it Integrality of $v$-adic multiple zeta values}, Publ. Res. Inst. Math. Sci. (59), No. 1, 123-151 (2023).


\bibitem[CM19]{CM19}
C.-Y.\ Chang and Y.\ Mishiba, \textit{On multiple polylogarithms in positive characteristic $p$: $v$-adic vanishing versus $\infty$-adic Eulerianness}, Int. Math. Res. Not. IMRN(2019), no. 3, 923--947.


\bibitem[CM21]{CM21}
C.-Y.\ Chang and Y.\ Mishiba, \textit{On a conjucture of Furusho over function fields}, Invent. math. \textbf{223}, (2021), 49--102.




\bibitem[CPY11]{CPY11}
C.-Y.\ Chang, M.\ A.\ Papanikolas and J.\ Yu, \textit{Frobenius difference equations and algebraic independence of zeta values in positive equal characteristic}, Algebra \& Number Theory \textbf{5} (2011), 111-129.


\bibitem[CY07]{CY07}
C.-Y.\ Chang and J.\ Yu, \textit{Algebraic relations among special zeta values in positive characteristic}, Adv. Math. \textbf{216} (2007), 321-345.



\bibitem[Col82]{Col82}
R.\ F.\ Coleman, \textit{Dilogarithms, regulators and p-adic L-functions}, Invent. Math. 69 (1982), no. 2, 171-208.

\bibitem[Fur04]{Fur04}
H.\ Furusho, \textit{$p$-adic multiple zeta values. I. $p$-adic multiple polylogarithms and the $p$-adic KZ equation}, Invent. Math. 155 (2004), no. 2, 253-286.



\bibitem[Gos79]{Gos79}
D.\ Goss, \textit{$v$-adic zeta functions, $L$-series, and measures for function fields. With an addendum}, Invent. Math. 55 (1979), no. 2, 117-119.


\bibitem[Gos96] {Gos96}
D.\ Goss, {\em Basic structures of function field
arithmetic}, Springer-Verlag, Berlin, 1996.

\bibitem[Gon97] {Gon97}
A.\ B.\ Goncharov, \textit{The double logarithm and Manin’s complex for modular curves}, Math. Res. Lett., \textbf{4} (1997), 617-636.













\bibitem[Neu13]{Neu13}
J. Neukirch, \textit{Algebraic Number Theory}, Vol. 322. Springer Science \& Business Media, 2013.

\bibitem[Pap08]{Pap08}
M. A. Papanikolas, \textit{Tannakian duality for Anderson-Drinfeld motives and algebraic independence of Carlitz logarithms}, Invent. Math. \textbf{171} (2008), no.~1, 123--174.






\bibitem[Sti93]{Sti93}
H.\ Stichtenoth, \textit{Algebraic Function Fields and Codes}, Springer-Verlag, Berlin, 1993.

\bibitem[Tha04]{Tha04}
D.\ S.\ Thakur, \textit{Function field arithmetic}, World Scientific Publishing, River Edge NJ, 2004.

\bibitem[Tha09]{Tha09}
D.\ S.\ Thakur, \textit{Power sums with applications to multizeta and zeta zero distribution for $\mathbb{F}_q[t]$}, Finite Fields Appl. \textbf{15} (2009), no. 4, 534--552.

\bibitem[Tha10]{Tha10}
D.\ S.\ Thakur, \textit{Shuffle relations for function field multizeta values}, Int. Math. Res. Notices IMRN(2010), no.11, 1973--1980.










\bibitem[Yam10]{Yam10}
G. Yamashita, \textit{Bounds for the dimensions of p-adic multiple L-value spaces}, Doc. Math. 2010, Extra vol.: Andrei A. Suslin sixtieth birthday, 687-723.

\bibitem[Yu91]{Yu91}
J. Yu, \textit{Transcendence and special zeta values in characteristic $p$}, Ann. of Math. (2) \textbf{134} (1991), no.~1, 1--23.



\bibitem[Zha16]{Zha16} J. Zhao, \textit{Multiple zeta functions, multiple polylogarithms and their special values}, Series on Number Theory and its Applications, 12. World Scientific Publishing Co. Pte. Ltd., Hackensack, NJ, 2016.

\end{thebibliography}

\end{document}